\documentclass{article}
\usepackage[utf8]{inputenc}
\usepackage[all,2cell]{xy}

\usepackage{amsmath}
\usepackage{amsthm}
\usepackage{amsfonts}
\usepackage{amssymb}
\newcommand{\Dial}{\mathfrak{Dial}(\mathsf{C})}
\newcommand{\mC}{\mathsf{C}}
\newcommand{\mD}{\mathsf{D}}

\newcommand{\mB}{\mathsf{B}}
\newcommand{\mE}{\mathsf{E}}
\newcommand{\Coprod}{\mathfrak{Sum}}
\newcommand{\Prod}{\mathfrak{Prod}}
\newcommand{\Dialmon}{\mathfrak{Dial}}
\def\skoledoc{\xymatrix{ \mathbin{Pred}: \mathbf{T}^{op}\ar[r]& \mathsf{Pos}}}

\newcommand{\fibration}[3]{#2 \colon #1 \longrightarrow #3}
\newcommand{\opfibration}[3]{#2^{\op} \colon #1^{(\op)} \longrightarrow #3}

\newcommand{\arrow}[3]{#1 \xrightarrow{#2} #3}

\def\mE{\mathsf{E}}
\def\mB{\mathsf{B}}
\def\mp{\mathsf{p}}
\def\mq{\mathsf{q}}
\def\mC{\mathsf{C}}
\def\mD{\mathsf{D}}
\def\id{\operatorname{id}}
\def\op{\operatorname{op}}
\def\dom{\operatorname{dom}}
\def\x{\times}
\def\pr{\pi}

\newcommand{\ovln}[1]{\overline{#1}}
\newcommand{\clift}[1]{#1^{\ast}}

\newcommand{\pullbackcorner}[1][ul]{\save*!/#1+1.2pc/#1:(1,-1)@^{|-}\restore}
\newcommand{\pullback}[8]{ \xymatrix{ 
#1 \pullbackcorner \ar[r]^{#5} \ar[d]_{#6} & #2 \ar[d]^{#7} \\
#3 \ar[r]_{#8} & #4 
}}

\newcommand{\coprodcomp}[1]{\mathfrak{Sum}(#1)}
\newcommand{\prodcomp}[1]{\mathfrak{Prod}(#1)}
\newcommand{\dial}[1]{\mathfrak{Dial}(#1)}
\newcommand{\prodmonad}{\mathfrak{Prod}}
\newcommand{\coprodmonad}{\mathfrak{Sum}}
\newcommand{\catfib}{\mathfrak{Cfib}}
\newcommand{\bemph}[1]{\textbf{\emph{#1}}}
\newcommand{\angbr}[2]{\langle #1,#2 \rangle}

\newcommand{\alexandre}[1]{\mathsf{Gr}_{#1}}

\def\infsl{\operatorname{\mathbf{InfSL}}}

\def\pos{\operatorname{\mathbf{Pos}}}

\def\existential{\operatorname{ex}}

\def\cartcat{\operatorname{\mathbf{CartCat}}}

\def\set{\operatorname{\mathbf{Set}}}

\def\coprodsym{\mathfrak{Sum}}

\def\radjunction{\;\hat{\dashv}\;}

\def\ladjunction{\;\check{\dashv}\;}

\def\ev{\textnormal{ev}}

\newcommand{\quadratocomm}[8]{ \xymatrix@+1pc{ 
#1 \ar[r]^{#5} \ar[d]_{#6} & #2 \ar[d]^{#7} \\
#3 \ar[r]_{#8} & #4 
}}

\usepackage[margin=1.6in]{geometry}

\usepackage{hyperref}
\usepackage{enumitem}

\theoremstyle{plain} 
\newtheorem{theorem}{Theorem}[section]
\newtheorem{cor}[theorem]{Corollary}
\newtheorem{lemma}[theorem]{Lemma}
\newtheorem{proposition}[theorem]{Proposition}

\theoremstyle{definition} 
\newtheorem{definition}[theorem]{Definition}
\newtheorem{remark}[theorem]{Remark}
\newtheorem{example}[theorem]{Example}

\title{The G\"odel Fibration}

\author{Davide Trotta \\ \small University of Pisa, IT \\ \small \textit{trottadavide92@gmail.com} \and Matteo Spadetto \\ \small University of Leeds, UK \\ \small \textit{matteo.spadetto.42@gmail.com} \and Valeria de Paiva \\ \small Topos Institute, CA \\ \small \textit{valeria@topos.institute}}
\begin{document}

\maketitle
\begin{abstract}
We introduce the notion of a G\"odel fibration, which is a fibration categorically embodying both the logical principle of traditional Skolemization (we can exchange the order of quantifiers paying the price of a functional)
and the existence of a prenex normal form presentation for every logical formula. Building up from Hofstra's earlier fibrational characterization of the de Paiva's categorical Dialectica construction, we show that a fibration is an instance of the Dialectica construction if and only if it is a G\"odel fibration. This result establishes an internal presentation of the Dialectica construction. Then we provide a deep structural analysis of the Dialectica construction producing a full
description of which categorical structure behaves well with respect to this construction,  focusing on (weak) finite products and coproducts. We conclude describing the applications we envisage for this generalized fibrational version of the Dialectica construction.
\end{abstract}
\section*{Introduction}
Historically, the Dialectica interpretation was devised by G\"odel~\cite{goedel1986} to prove the consistency of arithmetic. The interpretation allowed him to reduce the problem of proving the consistency of first-order arithmetic to the problem of proving the consistency of a simply-typed system of computable functionals, the well-known \emph{System T}. The key feature of the translation is that it turns formulae of arbitrary quantifier complexity into formulae of the form $\exists x \forall y \alpha (x,y)$.

Over the years, several authors have  explained the Dialectica interpretation in categorical terms. In particular, de Paiva~\cite{dePaiva1989Dialectica} introduced the notion of \emph{Dialectica categories} as an internal version of G\"odel's Dialectica Interpretation. The idea is to construct a category $\Dial$ from a category $\mC$ with finite limits.  The main focus in de Paiva's original work is on the categorical structure of the category $\Dial$ obtained, as this notion of a Dialectica category turns out to be also a model of Girard's Linear Logic~\cite{girard1987}.

This construction was first generalized by Hyland, who investigated the Dialectica construction associated to a fibred preorder \cite{HYLAND200243}. Later   Biering in her PhD work \cite{Biering_Dialecticainterpretations}  studied the Dialectica construction for an arbitrary cloven fibration.  

Meanwhile Hofstra~\cite{hofstra2011}  wrote 
an exposition and interpretation of the Dialectica construction from a modern categorical perspective, emphasizing its universal properties. 
His work gives  centre stage to the well-known concepts of monads, simple products and co-products. We take Hofstra's work as the basis for our work here. 
%

Hofstra shows that the original Dialectica construction $\Dialmon$ can be seen as the composition of two free constructions $\Coprod$ and $\Prod$, which are the simple co-product and product completions, respectively.
These completions are fully dual, so we only need to study one and can then deduce results for the other construction. However the whole Dialectica construction is not fully dual, as indicated by the order of the composition of the completions. 
Our work explains when the Dialectica construction can be performed, which hypotheses are necessary for the categorical construction, which properties of the construction  are preserved and why. Most importantly we are able to connect these preservation properties to the logic of the original interpretation, leading up to the definition of what we call the G\"odel fibration.

\section*{Our contributions}
The main contributions of this paper are the following.
\begin{enumerate}[label= \arabic{enumi}.,ref= \arabic{enumi}, wide=0pt]

\item \textit{We formalize the notion of fibrational quantifier-free formula.} Given Hofstra's characterization it is clear that instances of the Dialectica construction should have simple products and co-products, as the construction introduces completions under these. What else is necessary to get us a Dialectica construction? The first novelty of this work is the characterization of `covering quantifier-free objects' of a fibration. These objects correspond to formulas in the logic system that are quantifier-free. Of course, as usually happens in a categorical framework, such a syntactical notion of 'being quantifier-free' needs to be formalized in terms of a universal property.
The logical intuition behind our definition, is that an element $\alpha$ of a fibration $\mp$ is called \emph{quantifier-free} if it satisfies the following universal property, expressed in the internal language of $\mp$: if there is a proof $\pi$ of a statement $\exists i \beta(i)$ assuming $\alpha$, then there exists a \emph{witness} $t$, which depends on the  proof $\pi$, together with a proof of $\beta (t)$. Moreover, this must hold for every re-indexing $\alpha (f)$, because in logic if $\alpha (x)$ is quantifier-free then $\alpha (x)[f/x]=\alpha(f)$ is quantifier-free.
The covering requirement, as usual, means that, for every formula of the form $ i:I \;|\; \alpha (i)$, there exists a formula $\beta (i,a,b)$ quantifier-free  that is provably equivalent to it $\alpha (i) \dashv \vdash \exists a \forall b \beta(i,a,b)$. 

Notice that these requirements reflect G\"odel's original translation  and, at the same time, they recall standard conditions used in category theory to say that a category is \emph{free} for a given structure. 
One could think for example about the condition of having enough projectives in the exact completion of Carboni~\cite{CARBONI199879}. 

\item \textit{We introduce the notion of a G\"odel fibration.} 
A G\"odel fibration is a fibration with simple products and simple co-products, which, most importantly, admits a class of formal sub-objects which are \emph{free from products and co-products} and \emph{cover} all the elements of the fibre. Then we prove that a G\"odel fibration is a fibration categorically embodying both the logical principle of traditional Skolemization and the existence of a prenex normal presentation for every logical formula.

\item \textit{We provide a better definition of a Dialectica category.} We prove that a given fibration  is an instance of the Dialectica construction if and only if it is a G\"odel fibration. This result allows us to give a better definition of a Dialectica category because it shows which properties an arbitrary category should satisfy to be an instance of the Dialectica construction. In other words, given a fibration $\mp$ there exists a fibration $\mp'$ such that $\mp\cong \dial{\mp'}$ if and only if $\mp$ is a G\"odel fibration.
From a categorical perspective, we have classified the free-algebras for the Dialectica monad.

\item \textit{We prove that a Dialectica fibration satisfies a strong constructive feature in terms of witnesses.} 
We have shown that in the internal language of a Dialectica fibration, i.e. in the logic theory that canonically corresponds to this categorical notion, if there is a proof $\pi$ of a statement $\exists i\; \alpha(i)$, then there exists a \emph{witness} $t$, which depends on the  proof $\pi$, together with a proof of $\alpha (t)$. This principle is sometimes called the Rule of Choice. For example, Regular Logic (\url{https://ncatlab.org/nlab/show/regular+logic}) satisfies this principle, see \cite{trottamaietti2020}. 
\item \textit{We investigate which categorical structures are preserved by the Dialectica monad.}
We focus our attention on fibred (weak) finite products, fibred (weak) finite co-products and their \textit{associativities}. Using  Hofstra's decomposition  of
the monad $\Dialmon$ into its two parts, simple coproducts and products $\Coprod(\Prod)$,
we  prove first that the preservation of fibred weak finite products and coproducts together with their associativities depends on the existence of left and right \emph{weak adjoints} to re-indexing functors along its injections in the starting fibration, and \emph{points} in its base category. This base category is assumed to be a distributive category.
These hypotheses are quite natural since, for example, the sub-object fibration on a category with \emph{well-behaved} sums and points, e.g. $\set$, satisfies all of these requirements.   Moreover, these assumptions are also satisfied by the resulting fibration. Under these hypotheses we have shown that the Dialectica construction preserves fibred weak finite products and co-products. Moreover, we show that, under these assumptions, the total category $\dial{\mp}$ has finite weak products and coproducts. 
\end{enumerate}
%

\section*{Related work}


In this work we combine the generalization of the Dialectica construction presented by Hofstra in \cite{hofstra2011}, with the structural analysis of the Dialectica construction due to Hyland \cite{HYLAND200243} and Biering \cite{BIERING2008290}. 
The works of Hyland and Biering extend the results provided by de Paiva in \cite{dePaiva1989Dialectica} to the fibrational setting, in particular, they provide sufficient conditions to make the Dialectica category associated to a cloven fibration into a category with finite products.
Our framework 
shares some of their assumptions, as we are also aiming for a total category with products.  We require the existence of right-weakly left adjoints and left-weakly right adjoints to the re-indexing functors (along the injections in the basis of a given fibration). This is the main ingredient of our notion of an \textit{extendable fibration} (see Definition \ref{definition godel fibration}). These fibrations will provide us with (weak) products and coproducts in the total category.

Another work which deals with structures preserved by Dialectica-like constructions, is the work of Moss and von Glehn \cite{moss-vonglehn2018}, where the authors are interested not in the original construction, but in a modified version of G\"odel's Dialectica interpretation for models of intensional Martin-L\"of type-theory, through the notion of fibred display map category. Their work focus on the preservation of the type constructors, while they drop the layer of predicates from their Dialectica propositions, considering only those Dialectica propositions of the form $\exists x\forall y\top$. In fact, they call their construction the \emph{polynomial model}, explaining that this name fits better, because they are considering the \emph{predicate-free} Dialectica construction.

Finally, Topos-and tripos-theoretic versions of the Dialectica construction have been studied by Biering in \cite{Biering_Dialecticainterpretations}, while the recent work of Shulman \cite{Shulman2020}, describes a "polycategorical" version of a generalization of the Dialectica construction.
Other applications of completions involving  universal and existential quantifiers can be found in  \cite{trottamaietti2020,trotta2020,Frey2014AFS}, where similar constructions are presented in the language of doctrines.

\section{Logical Fibrations}
The basic reference to connecting logic and fibrations in Jacob's book \cite{Jacobs1999}.
We recall the basic facts that we need from the book in the appendix.

But in general we say  that a fibration has \bemph{fibred $\diamond$'s} if all fibre
categories have $\diamond$'s and substitution functors preserve $\diamond$'s.

We review some necessary background on fibrations and completions. 
\begin{definition}
Let $\fibration{\mE}{\mp}{\mB}$ be a functor and $\arrow{X}{f}{Y}$ an arrow in $\mE$. Let us call $\arrow{A}{u:=\mp (f)}{B}$ the arrow $\mp(f)$ of $\mB$. We say that $f$ is \bemph{Cartesian over} $\boldsymbol{u}$ if, for every morphism $\arrow{Z}{g}{Y}$ in $\mE$ such that $\mp (g)$ factors through $u$, $\mp (g)=u w$, there exists a unique $\arrow{Z}{h}{X}$ of $\mE$ such that $g=f h$ and $\mp (h)=w$.
\end{definition}

\begin{definition}
A \bemph{fibration} is a functor $\fibration{\mE}{\mp}{\mB}$ such that, for every $Y$ in $\mE$ and every $\arrow{I}{u}{\mp Y}$, there exists a Cartesian arrow  $\arrow{X}{f}{Y}$ over $u$.
\end{definition}
For a given fibration $\fibration{\mE}{\mp}{\mB}$, and for any $A$ in $\mB$, let $\mE_A$ be the \bemph{fibre category} over $A$: its objects are the objects $X$ of $\mE$ such that $\mp X=A$, and its morphisms, which are said to be \bemph{vertical}, are the morphisms $\arrow{X}{f}{Y}$ of $\mE$ such that $\mp (f)=\id_A$.

Recall from \cite{Jacobs1999} that a fibration is called \bemph{cloven} if it comes together with a choice of Cartesian liftings (a choice of cartesian liftings is called a \emph{cleavage}); and it is called \bemph{split} if it is cloven and the given liftings
are well-behaved in the sense that they satisfy certain functoriality conditions.

In case a fibration $\fibration{\mE}{\mp}{\mB}$ is equipped with a cleavage, for every morphism $\arrow{B}{u}{\mp(Y)}$ of $\mB$ we denote the chosen cartesian lifting of $u$ by $\arrow{u^\ast (Y)}{\ovln{u}Y}{Y}$. Then we can define the \bemph{substitution functor}: $$ \arrow{\mE_B}{u^\ast}{\mE_A}$$ sending $X$ to $u^\ast(X)$ and a vertical morphism $\arrow{X}{f}{Y}$ to the unique mediating map $u^\ast f$ in:
$$ \xymatrix{
u^\ast(X) \ar@{-->}[d]_{u^\ast f} \ar[r]^{\ovln{u}X} & X \ar[d]^f \\
u^\ast(Y) \ar[r]_{\ovln{u}Y} & Y.
}$$


\begin{proposition}[Prop. 9.2.1 \cite{Jacobs1999}]
Let $\fibration{\mE}{\mp}{\mB}$ be a fibration with finite products in its base category $\mB$. The $\mp$ has fibred finite product if and only if the total category $\mE$ has finite products via fibred functors, and the functor $\mp$ strictly preserves these products.
\end{proposition}

One of the pillars of categorical logic is  Lawvere's 
crucial intuition which considers logical languages and theories as indexed categories
and studies their 2-categorical properties. In this setting connectives and quantifiers are characterized in terms of adjointness relations \cite{lawvere1969,lawvere1969b,lawvere1970}.

In this fibrational setting, the intuition is that the base category $\mB$ of a fibration $\fibration{\mE}{\mp}{\mB}$ represents the category of (type-theoretical) contexts, a fibre $\mE_I$ represents the propositions $\alpha(i)$ in the context $I$, and the morphisms are proofs. Cartesian morphisms of $\mp$ induce a \emph{re-indexing} or substitution operation.
From this perspective, the simple form of quantification is described in terms of adjoints to weakening functors $\pr^*$ along projections $\pi$. For example, the existential quantification is given by an operation $ \coprod_{\pr^*}: \mE_{A\x B} \rightarrow \mE_{A}$, which sends a proposition $\alpha (a,b)$ to $\exists b\; \alpha (a,b)$. 

Now we briefly recall the formal definition of a fibration with simple products (or simple
universal quantification) and coproducts (or simple
existential quantification). For  a complete presentation of the theory of fibrations and its connection to type theory, we refer the reader to \cite{Jacobs1999}. 
In this work, we will assume that a fibration $\mp$ is always \emph{cloven and split}, i.e. that the re-indexing operation is functorial. 
\footnote{These definitions can be found in pages 47 and 49 of \cite{Jacobs1999}.}

\begin{definition}
A fibration $\fibration{\mE}{\mp}{\mB}$ with a cleavage over a category $\mB$ with finite products has \bemph{simple coproducts} when all \emph{weakening functors}, i.e. the Cartesian lifitngs of projections, have
left adjoints satisfying \bemph{Beck-Chevalley Condition} (BCC) for pullback squares
of the form:
$$\pullback{I\x X}{I}{J\times X}{J.}{\pr_I}{f\x \id_X}{f}{\pr_J}$$
Dually, a fibration $\fibration{\mE}{\mp}{\mB}$ has  \bemph{simple products}, when weakening functors have right adjoints satisfying BCC.
\end{definition}
The next concept we are going to introduce is that of \bemph{opposite} of a fibration. Recall from \cite[Lemma 1.4.10]{Jacobs1999} that, given a fibration $\fibration{\mE}{\mp}{\mB}$, for every cleavage of $\mp$ one has the isomorphism of sets (or classes): 
$$ \begin{aligned}
\mE(X,Y)&\overset{\cong}{\to} \coprod_{u:\mp X \rightarrow \mp Y} \mE_{\mp X}(X, u^{\ast}(Y)) \\
f&\mapsto (\mp f, f')
\end{aligned}$$
where $\coprod$ is the disjoint union and $f'$ is the unique vertical arrow such that $f=(\overline{ \mp f}(\mp Y))f'$. This means that a morphism in a total category $\mE$ corresponds to a morphism in the basis together with a vertical map. The intuition behind the definition of the opposite is that all vertical maps in such composites are reversed. 

\begin{definition}[Jacobs \cite{Jacobs1999}]
Let $\fibration{\mE}{\mp}{\mB}$ be a fibration.  We describe a new fibration on the same basis written as $\opfibration{\mE}{\mp}{\mB}$, which is fibrewise the opposite of $\mp$.
\end{definition}
Let $CV$ be the class:
$$\{ (f_1,f_2)| f_1\mbox{ is Cartesian, } f_2 \mbox{ is vertical and} \dom (f_1)=\dom (f_2)\}.$$ An equivalence realtion is defined on the collection $CV$ by $(f_1,f_2)\sim (g_1,g_2)$ if there exists an arrow $h$ such that $f_1=g_1h$ and $f_2=g_2h$. The equivalence class of $(f_1,f_2)$ is denoted by $[f_1,f_2]$.
The total category $\mE^{(\op)}$ has the same objects of $\mE$, and morphisms $X\rightarrow Y$ are equivalence classes $[f_1,f_2]$ of arrows $f_1$ and $f_2$, as in
$$\xymatrix{ X\\
\bullet \ar[r]_{f_1} \ar[u]^{f_2}&Y
}$$
The composition $[g_1,g_2]\circ [f_1,f_2]$ is described by the following diagram:
$$\xymatrix{ X\\
A \ar[rr]^{f_1} \ar[u]_{f_2}&&Y\\
(\mp f_1)^*B\ar@{.>}[u]^{x} \ar@{.>}[rr]_>>>>>>>>>>>{\ovln{\mp f_1} B} && B \ar[rr]^{g_1} \ar[u]_{g_2} && Z }$$
where $x$ is the unique vertical arrow arising by cartesianity of $f_1$ and making the diagram commute. We define the composition $[g_1,g_2]\circ [f_1,f_2]$ to be the class: $$[g_1(\ovln{\mp f_1} B),f_2x]$$ which turns out to be well-defined. See \cite[Definition 1.10.11]{Jacobs1999} for more details.

The functor $\opfibration{\mE}{\mp}{\mB}$ is defined by the assignments $X\mapsto \mp X$ and $[f_1,f_2]\mapsto \mp (f_1)$, It is well-defined since $f_2$ is vertical.

\begin{lemma}
Let $\fibration{\mE}{\mp}{\mB}$ be a fibration. Then it is the case that:
\begin{enumerate}
    \item $\opfibration{\mE}{\mp}{\mB}$ is a fibration, and a morphism $[f_1,f_2]$ is a Cartesian arrow if and only if $f_2$ is an isomorphism;
    \item for every object $A$ of the base category $\mB$, there is a natural isomorphism $\mE^{(\op)}_A\cong (\mE_A)^{\op}$, i.e. the fibration $\mp^{\op}$ is the fibrewise opposite of $\mp$.
\end{enumerate}
\end{lemma}
The following useful lemma says that, as in ordinary category theory one has that a category $\mC$ has finite limits if and only if $\mC^{\op}$ has finite colimits, a similar result exists for
fibred categories, since the opposite of a fibration is its \textit{fibrewise} opposite.
\begin{lemma}\label{opfibration}
Let $\fibration{\mE}{\mp}{\mB}$ be a fibration. Then it it the case that:
\begin{enumerate}
     \item there is an isomorphism of fibration $(\mp^{\op})^{\op}\cong \mp$;
    \item $\mp$ has fibered limits if and only if $\mp^{\op}$ has fibered colimits;
    \item $\mp$ has simple products if and only if $\mp^{\op}$ simple coproducts.
\end{enumerate}
\end{lemma}

\noindent
We conclude the current section by stating and proving the following:

\begin{lemma}\label{fully faithful}
Let $\fibration{\mE}{\mp}{\mB}$ and $\fibration{\mE'}{\mp'}{\mB}$ be fibrations and let $F$ be a morphism of fibrations $\mp \to \mp'$, i.e. $F$ is a functor $\mE \to \mE'$ such that $\mp'F=\mp$ and $F$ preserves cartesian morphisms over a given arrow of $\mB$.

Let us assume that $F$ is an equivalence on the fibres, i.e. $F\upharpoonright_{\mE_I} \colon \mE_I \to \mE_I'$ is an equivalence for every object $I$ of $\mB$. Then $F$ is fully faithful.

\proof
Let $X$ and $Y$ be objects of $E$. Then: $$\begin{aligned}
\mE(X,Y)&\cong \coprod_{u:\mp X \rightarrow \mp Y} \mE_{\mp X}(X, u^{\ast}(Y))\cong \coprod_{u:\mp X \rightarrow \mp Y} \mE'_{\mp X}(F X, F 
u^{\ast}(Y)) \\
&\cong\coprod_{u:\mp'F X \rightarrow \mp'F Y} \mE'_{\mp'F X}(FX, u^{\ast}(FY))\cong \mE'(FX,FY)
\end{aligned}$$ and we are done.
\endproof
\end{lemma}

We conclude this section by recalling the 2-categorical structure of fibrations. We define $\catfib$ the 2-category of cloven, split fibrations. The objects of $\catfib $ are cloven, split fibrations $\fibration{\mE}{\mp}{\mB}$. For fibrations $\fibration{\mE}{\mp}{\mB}$ and $\fibration{\mD}{\mq}{\mB'}$, a 1-cell is a pair $(F_0,F_1)$ of functors such that the diagram
\[\xymatrix{
\mD \ar[r]^{F_0} \ar[d]_{\mq} & \mE\ar[d]^{\mp}\\
\mB' \ar[r]_{F_1} & \mB
}\]
commutes,  where $F_0$ is a \emph{fibred functor}, i.e. it sends $\mq$-cartesian arrow to $\mp$-cartesian arrow. A 2-cell from $(F_0,F_1)$ to $(G_0,G_1)$ is a pair $( \phi_0,\phi_1)$ of natural transformations $\arrow{F_0}{\phi_0}{G_0}$, $\arrow{F_1}{\phi_1}{G_1}$ such that $i_{\mp}.\phi_0=\phi_1.i_{\mq}$, i.e. for every $X$ of $\mD$ we have $\mp ({\phi_0})_X=(\phi_1)_{\mq X}$.

\section{The dialectica monad}
In this section we assume that  $\fibration{\mE}{\mp}{\mB}$ is  a cloven and split fibration whose base category $\mB$ has finite products. First we recall from Hofstra's \cite{hofstra2011} the free construction which adds simple coproducts to a fibration, and then the dual construction which freely adds simple products.

\bigskip

\noindent
\textbf{Simple coproducts completion}. The category $\coprodcomp{\mp}$ has:
\begin{itemize}
    \item as \textbf{objects}  triples $(I,X,\alpha)$, where $I$ and $X$ are objects of the base category $\mB$ and $\alpha$ is an object of the fibre $\mE_{I\x X}$;
    \item as \textbf{morphisms}  triples $\arrow{(I,X,\alpha)}{(f_0,f_1,\phi)}{(J,Y,\beta)}$, where $\arrow{I}{f_0}{J}$ and $\arrow{I\x X}{f_1}{Y}$ are arrows of $\mB$ and $\arrow{\alpha}{\phi}{\angbr{f_0\pr_I}{f_1}^{\ast}(\beta)}$ is a morphism of the fibre category $\mE_{I\x X}$.
\end{itemize}

\noindent
The category $\coprodcomp{\mp}$ is fibred over $\mB$ via the first component projection and this fibration is denoted by $\fibration{\coprodcomp{\mp}}{\coprodcomp{\mp}}{\mB}$. This fibration is called  the \bemph{simple coproduct (or sum) completion} of $\mp$. The intuition behind this definition is that an object $(I,X,\alpha)$ of the fibre category $\coprodcomp{\mp}_{I}$ represents a formula $(\exists x:X) \alpha(i,x)$. The assignment $\mp\mapsto \coprodcomp{\mp}$ extends to a KZ pseudo-monad on the 2-category of cloven split fibrations, see \cite[Ttheorem 3.9]{hofstra2011}.

\begin{remark}[\textit{A presentation of $\coprodcomp{\mp}$ reindexing functors}] \label{reindexing}
Let $\fibration{\mE}{\mp}{\mB}$ be a cloven and split fibration. Let $\arrow{I}{f}{J}$ be an arrow of $\mB$ and let $(J,Y,\beta)$ be an object of $\coprodcomp{\mp}_J$. Then the triple: $$(\;\arrow{I}{f}{J},\;\arrow{I\times Y}{\pi_Y}{Y},\;\arrow{\langle f\pr_I,\pr_Y \rangle^*\beta}{1_{\langle f\pr_I,\pr_Y \rangle^*\beta}}{\langle f\pr_I,\pr_Y \rangle^*\beta}\;)$$ is $\coprodcomp{\mp}$-cartesian $(I,Y,\langle f\pr_I,\pr_Y \rangle^*\beta) \to (J,Y,\beta)$ over $\arrow{I}{f}{J}$. In particular $\coprodcomp{\mp}$ is endowed with a cloven and split structure. If: $$\arrow{(J,Y,\beta)}{(\arrow{J\times Y}{g}{Y'},\;\arrow{\beta}{\gamma}{\langle \pr_J,g\rangle^*\beta'})}{(J,Y',\beta')}$$ is an arrow of $\coprodcomp{\mp}_J$ (observe the omission of the first component, as it is forced to be the identity arrow on $J$) then its $f$-reindexing is the pair: $$\arrow{(I,Y,\langle f\pi_I,\pi_Y \rangle^*\beta)}{(g\langle f\pi_I,\pi_Y\rangle,\;\langle f\pi_I,\pi_Y\rangle^*\gamma )}{(I,Y',\langle f\pi_I,\pi_{Y'} \rangle^*\beta')}$$ of $\coprodcomp{\mp}_I$, whose first component is the arrow $\arrow{I\times Y}{g\langle f\pi_I,\pi_Y \rangle}{Y'}$ of $\mB$ and whose second one is the arrow: $$\arrow{\langle f\pi_I,\pi_Y \rangle^*\beta}{\langle f\pi_I,\pi_Y \rangle^*\gamma}{\langle f\pi_I,\pi_Y \rangle^*\langle \pr_J,g\rangle^*\beta'=\langle \pi_I,g\langle f\pi_I,\pi_Y \rangle\rangle^*\langle f\pi_I,\pi_{Y'} \rangle^*\beta'}$$ of $\mE_{I\times Y}$.

Now, let us assume that $f$ is a projection $\arrow{J \times K}{\pr_J}{J}$. In this particular case (in which we are mostly interested) such an annoying presentation \textit{collapses} into the following easier one: the $\pr_J$-weakening of the arrow $(g,\gamma)$ of $\coprodcomp{\mp}_J$ is the pair: $$\arrow{(J\times K,Y,\langle \pi_J,\pi_Y \rangle^*\beta)}{(g\langle \pi_J,\pi_Y\rangle,\langle \pi_J,\pi_Y\rangle^*\gamma)}{(J\times K,Y',\langle \pi_J,\pi_{Y'}\rangle^*\beta')} $$ of $\coprodcomp{p}_I$, whose first component is the arrow $\arrow{J\times K \times Y}{g \langle \pi_J,\pi_Y\rangle}{Y'}$ of $\mB$ and whose second one is the arrow: $$\arrow{\langle \pi_J,\pi_Y\rangle^*\beta}{\langle \pi_J,\pi_Y\rangle^*\gamma}{\langle \pi_J,\pi_Y\rangle^*\langle \pi_J,g \rangle^*\beta'=\langle \langle\pi_J,\pi_K \rangle, g\langle\pi_J,\pi_Y \rangle\rangle^*\langle \pi_J,\pi_{Y'}\rangle^*\beta'}$$  of $\mE_{J\times K\times Y}$.
\end{remark}

\begin{remark}[\textit{$\coprodcomp{\mp}$ has simple coproducts}] \label{coprodcomp has simple coprod} Let $\mp$ be a cloven and split fibration and let us consider a projection $\arrow{J\times K}{\pi_J}{J}$ of $\mB$. The left adjoint $\coprod_{\pi_J}$ of the $\pi_J$-weakening $\pi_J^*$ in $\coprodcomp{p}$ (see Remark \ref{reindexing}) exists and sends an arrow: $$\arrow{(J \times K, Y, \beta)}{(\arrow{J\times K\times Y}{g}{Y'},\;\arrow{\beta}{\gamma}{\langle \langle \pi_J,\pi_K\rangle,g\rangle^*\beta'})}{(J\times K,Y',\beta')}$$ of $\coprodcomp{\mp}_{J\times K}$ to the arrow: $$\arrow{(J, K\times Y, \beta)}{(\arrow{J\times K\times Y}{\langle \pr_K,g \rangle}{K\times Y'},\;\arrow{\beta}{\gamma}{\langle  \pi_J,\langle \pi_K,g\rangle\rangle^*\beta'})}{(J, K\times Y',\beta')}$$ of $\coprodcomp{\mp}_{J}$, which we also denote as: $$\arrow{\coprod_{\pr_J}(J\times K,Y,\beta)}{\coprod_{\pr_J}(g,\gamma)}{\coprod_{\pr_J}(J\times K,Y',\beta')}.$$
\end{remark}

\begin{remark}
Let $\fibration{\mE}{\mp}{\mB}$ be a fibration and consider its simple coproduct completion $\fibration{\coprodcomp{\mp}}{\coprodcomp{\mp}}{\mB}$. As a consequence of Remark \ref{coprodcomp has simple coprod}, every element $(I,A,\alpha)$ of the fibre $\coprodcomp{\mp}_I$ equals the object $\coprod_{\pr_I}(I\x A,1,\alpha)$.
\end{remark}

Notice that, by dualising the previous construction, one gets the notion of simple product completion together with its analogous version of the previous characterization.

\medskip

\noindent
\textbf{Simple products completion}. The category $\prodcomp{\mp}$ is the one:
\begin{itemize}
    \item whose \textbf{objects} are triples $(I,X,\alpha)$, where $I$ and $X$ are objects of the base category $\mB$ and $\alpha$ is an object of the fibre $\mE_{I\x X}$;
    \item whose \textbf{morphisms} are triples $\arrow{(I,X,\alpha)}{(f_0,f_1,\phi)}{(J,Y,\beta)}$, where $\arrow{I}{f_0}{J}$ and $\arrow{I\x Y}{f_1}{X}$ are arrows of $\mB$ and $\arrow{\clift{\angbr{\pr_I}{f_1}}(\alpha)}{\phi}{\clift{\langle f_0\pi_I, \pi_Y\rangle}(\beta)}$ is a morphism of the fibre category $\mE_{I\x X}$.
\end{itemize}

\noindent
Again, the category $\prodcomp{\mp}$ is fibred over $\mB$ via first component projection and this fibration is denoted by $\fibration{\prodcomp{\mp}}{\prodcomp{\mp}}{\mB}$ and called \bemph{simple product completion} of $\mp$. The intuition behind this definition is that an object $(I,X,\alpha)$ of the fibre category $\prodcomp{\mp}_{I}$ represents a formula $(\forall x:X) \alpha(i,x)$.

\begin{proposition}\label{proposition coproc is dual to prod}
There is an isomorphism of fibrations: $$\prodcomp{\mp}\cong \coprodcomp{\mp^{\op}}^{\op}$$ which is natural in $\mp$.
\end{proposition}
Again, the assignment $\mp\mapsto \coprodcomp{\mp}$ extends to a co-KZ pseudo-monad on the 2-category of cloven split fibrations, and its 2-category of pseudo-algebras is equivalent to the 2-category of fibrations with simple products, see \cite[Theorem 3.12]{hofstra2011}.

We conclude this section by recalling the presentation of the dialectica construction and its presentation via the product-coproduct completions.

\medskip

\noindent
\textbf{Dialectica construction}. Let $\fibration{\mE}{\mp}{\mB}$ be a fibration. Define a category $\dial{\mp}$ as follows:
\begin{itemize}
\item \bemph{objects} are quadruples $(I, X,U,\alpha)$ where $I,X$ and $U$ are objects of the base category $\mB$ and $\alpha\in \mE_{I\times X\times U}$ is an objects of the fibre of $\mp$ over $I\times X\times U$;
\item a \bemph{morphism} from $(I, X,U,\alpha)$  to $(J, Y,V,\beta)$ is a quadruple $(f,f_0,f_1,\phi)$ where
\begin{enumerate}
    \item $\arrow{I}{f}{J}$ is a morphism in $\mB$;
    \item $\arrow{I\times X}{f_0}{Y}$ is a morphism in $\mB$;
    \item $\arrow{I\times X\times V}{f_1}{U}$ is a morphism in $\mB$;
    \item $\arrow{\alpha(i,x,f_1(i,x,v))}{\phi}{\beta (f(i),f_0(i,x),v)}$ is an arrow in the fibre over $I\times X\times V $.
\end{enumerate}
\end{itemize}
This is a fibration on $\mB$ with the projection on the first component. The key observation of Hofstra is that the construction of the fibration $\dial{\mp}$ can be decomposed in two steps.

\begin{lemma}[Hofstra \cite{hofstra2011}]
There is an isomorphism of fibrations, natural in $\mp$:
$$\dial{\mp} \cong \coprodcomp{\prodcomp{\mp}}.$$ 
\end{lemma}
Notice that the pseudo-functor $\coprodcomp{\prodcomp{-}}$ is not a pseudo-monad in general, but,
in the case the base category $\mB$ of a fibration $\fibration{\mE}{\mp}{\mB}$ is cartesian closed, one can shows that there exists a pseudo-distributive law $$\arrow{\prodmonad\coprodmonad}{\lambda}{\coprodmonad\prodmonad}$$ 
of pseudo-monads, see \cite[Theorem 4.4]{hofstra2011}. Therefore, by the known equivalence between lifting and pseudo-distributive laws, see for example \cite{tanaka2004,tanaka2006}, in this case we have that $\coprodcomp{\prodcomp{-}}$ is a pseudo-monad.

An notably advantage of this algebraic presentation of the dialectica construction, is that the principle of Skolemisation is represented by the pseudo-distributive law $\lambda$. 
\begin{theorem}
When the base category $\mB$ of a fibration $\mp$ is cartesian closed, the fibration $\dial{\mp}$ stisfy the principle 
\[ \forall x \exists y \alpha(i,x,y)\cong \exists f \forall x \alpha (i,x,fx)\]
for every $\alpha$.
\end{theorem}

\section{The notion of Skolem and G\"odel fibrations}
When one deals with quantification, for example in first-order logic, it is very common to assert something like \emph{a formula $\alpha$ is quantifier-free}. This assertion has a natural meaning from a syntactic point of view, but it is not clear how it should be presented from a categorical perspective.
The aim of the following definitions, which are generalizations of definitions in \cite{trottamaietti2020} to the fibrational setting, is to capture the common property of those elements of a given fibration $\fibration{\mE}{\mp}{\mB}$ which will appear as quantifier-free propositions in the internal language of the fibration $\mp$.
We start by defining when an element of a fibre of $\mp$ is free from the existential quantifier, and then we dualize the definition for the universal quantifier. (Recall that the symbols $\coprod$ and $\prod$ represent the logical quantifiers $\exists$ and $\forall$.)

The logical intuition behind the next definition is that an element $\alpha$ is \emph{quantifier-free} if it satisfies the following universal property: if there is a proof $\pi$ of a statement $\exists i\; \beta(i)$ assuming $\alpha$, then there exists a \emph{witness} $t$, which depends on the  proof $\pi$, together with a proof of $\beta (t)$. Moreover, we require that this holds for every re-indexing $\alpha (f)$ because in logic  quantifier-free propositions are stable under substitution, i.e. if $\alpha (x)$ is quantifier-free then $\alpha(f)$ is quantifier-free.

\begin{definition}\label{definition P-atomic}
Let $\fibration{\mE}{\mp}{\mB}$ be a fibration with simple coproducts. 
An object $\alpha$ of the fibre $\mE_I$ is said to be $\coprod$\bemph{-quantifier-free} if it enjoys the following universal property. For every arrow $f$ and every projection $\pi_A$ in $\mB$ as follows: $$\xymatrix{ A\times B \ar[r]^>>>>>{\pr_A} &A \ar[r]^f &I}$$ and every vertical arrow: $$ \arrow{\clift{f}\alpha}{h}{ \coprod_{\pr_A}\beta} $$
of $\mE_A$, where $\beta$ is an object of the fibre $\mE_{A\times B}$, there exist a unique arrow $\arrow{A}{g}{B}$ of $\mB$ and a unique vertical arrow $\arrow{\clift{f}\alpha}{\ovln{h}}{\clift{\angbr{1_A}{g}}\beta}$ of $\mE_A$ such that $h$ equals the composition: $$\Big(\;\;\arrow{\clift{f}\alpha}{\ovln{h}}{\clift{\angbr{1_A}{g}}\beta}\xrightarrow{\clift{\angbr{1_A}{g}}\eta_{\beta}}\clift{\angbr{1_A}{g}}\big(\pi_A^*\coprod_{\pi_A}\beta\big)=\coprod_{\pi_A}\beta\;\;\Big)$$ where $\arrow{\beta}{\eta_\beta}{\pi_A^*\coprod_{\pi_A}\beta}$ is the unit at $\beta$ of the adjunction $\coprod_{\pr_A}\dashv \clift{\pr_A}$. 
\end{definition}

\medskip 

Clarifying the concrete meaning of Definition \ref{definition P-atomic}, the given object $\alpha$ of $\mE_I$ represents a formula $\alpha(i)$. Given an arrow $\arrow{A}{f}{I}$ a term $fa : I$ is in the context $a:A$, and it is the case that $f^*\alpha$ represents the corresponding formula $\alpha(fa)$. The object $\beta$ of $\mE_{A \times B}$ corresponds to a formula $\beta(a,b)$, the object $\coprod_{\pi_A}\beta$ represents the formula $(\exists b)\beta (a,b)$, which is  in the same context $a:A$ of $\alpha (fa)$. Meanwhile, the object $\langle 1_A,g\rangle^*\beta$ is again the re-indexing of $\beta(a,b)$ through an arrow $\arrow{A}{\langle 1_A,g\rangle}{A\times B}$, hence it represents the formula $\beta(a,ga)$, which is in the same context $a:A$ of $\alpha(fa)$ and $(\exists b)\beta (a,b)$.

Thus the property we require of the formula $\alpha(i)$ is the following: whenever there is a proof (an arrow $h$) of $(\exists b)\beta (a,b)$ from $\alpha(fa)$ (for some term $fa:I$ in the context $a:A$), then there is a unique term $ga:B$ in the context $a:A$ together with a unique proof $\overline{h}$ of $\beta(a,ga)$ from $\alpha(fa)$, in such a way that, adding at the end of the proof $\overline{h}$ the canonical proof of $(\exists b)\beta (a,b)$ from $\beta(a,ga)$ (which is represented by the re-indexing of the unit at $\beta$ of the adjunction $\coprod_{\pi_A}\dashv \pi_A^*$), we get back to the proof $h$ itself of $(\exists b)\beta (a,b)$ from $\alpha(fa)$. 

Observe that this is precisely the universal property, that we presented before Definition \ref{definition P-atomic}, enjoyed by a formula which is free from existential quantification. 
\begin{remark}\label{remark coproduct quant free form a subfibrations}
Notice that if we consider a fibration $\mp$ with simple coproducts, then one can define a sub-fibration $\mp'\rightarrow \mp$ such that the fibres of $\mp'$ are given by $\coprod$-quantifier-free objects, and the base category of $\mp'$ is the same of $\mp$. This follows since $\coprod$-quantifier-free objects are stable under re-indexing by definition.
\end{remark}

The next concept we are going to need in the categorical setting reminds that of \emph{existence of a prenex normal form} used in logic. Recall, for example from \cite{vandalen_logic}, that in (classical) FOL every formula is equivalent to some formula in prenex normal form.

\begin{definition}\label{definition enought-coprod-free-quantifaier}
We say that a fibration with simple coproducts $\fibration{\mE}{\mp}{\mB}$ has \bemph{enough} $\coprod$\bemph{-quantifier-free objects} if, for every object $I$ of $\mB$ and for every element $\alpha\in \mE_I$, there exist an object $A$ and a $\coprod$-quantifier-free object $\beta$ in $\mE_{I\times A}$ such that
$\alpha\cong\coprod_{\pr_I}\beta.$
\end{definition}

By duality we can define the same concept with respect to the universal quantifier.

\begin{definition}\label{definition P-atomic-product}
Let $\fibration{\mE}{\mp}{\mB}$ be a fibration with simple products. An object $\alpha$ of the fibre $\mE_I$ is said to be $\prod$\bemph{-quantifier-free} if it enjoys the following universal property: for every arrow $f$ and every projection $\pi_A$ in $\mB$ as follows: $$\xymatrix{ A\times B \ar[r]^>>>>>{\pr_A} &A \ar[r]^f &I}$$ and every vertical arrow: $$ \arrow{\prod_{\pr_A}\beta}{h}{ \clift{f}\alpha} $$
of $\mE_A$, where $\beta$ is an object of the fibre $\mE_{A\times B}$, there exist a unique arrow $\arrow{A}{g}{B}$ of $\mB$ and a unique vertical arrow $\arrow{\clift{\angbr{1_A}{g}}\beta}{\ovln{h}}{\clift{f}\alpha}$ of $\mE_A$ such that $h$ equals the composition: $$\Big(\;\;\prod_{\pi_A}\beta =\langle 1_A,g\rangle^*\big(\pi_A^*\prod_{\pi_A}\beta\big) \xrightarrow{\langle 1_A,g\rangle^* \varepsilon_\beta}\arrow{\langle 1_A,g \rangle^*\beta}{\ovln{h}}{f^*\alpha}   \;\;\Big)$$ where $\arrow{\pi_A^*\prod_{\pi_A}\beta}{\varepsilon_\beta}{\beta}$ is the counit at $\beta$ of the adjunction $\clift{\pr_A}\dashv \prod_{\pr_A}$. 
\end{definition}

\begin{definition}\label{definition enought-prod-free-quantifaier}
We say that a fibration with simple products $\fibration{\mE}{\mp}{\mB}$ has \bemph{enough}-$\prod$-\bemph{quantifier-free objects} if, for every object $I$ of $\mB$ and for every element $\alpha\in \mE_I$, there exist an object $A$ and a $\prod$-quantifier-free object $\beta$ in $\mE_{I\times A}$ such that
$\alpha\cong\prod_{\pr_I}(\beta).$
\end{definition}

Now we can introduce a particular kind of fibration called a \emph{Skolem fibration}. The name is chosen because these fibrations satisfy a version of the traditional principle of \emph{Skolemization}, as presented in \cite{goedel1986} and \cite{hofstra2011}.

\begin{definition}\label{definition skolem fibration}
A fibration $\fibration{\mE}{\mp}{\mB}$ is called a \bemph{Skolem fibration} if:
\begin{itemize}
    \item its base category $\mB$ is cartesian closed;
    \item the fibration $\mp$ has simple products and simple coproducts;
    \item the fibration $\mp$ has enough $\coprod$-quantifier-free objects.
    \item $\coprod$-quantifier-free objects are stable under simple products, i.e. if $\alpha\in \mE_I$ is a $\coprod$-quantifier-free object, then $\prod_{\pr}(\alpha)$ is a $\coprod$-quantifier-free object for every projection $\pr$ from $I$.
\end{itemize}
\end{definition}

Notice that the last point of Definition \ref{definition skolem fibration} implies that, given a Skolem fibration $\fibration{\mE}{\mp}{\mB}$, the sub-fibration $\fibration{\mE'}{\mp'}{\mB}$ of $\coprod$-quantifier-free objects of $\mp$ defined in Remark \ref{remark coproduct quant free form a subfibrations} has simple products.
\begin{proposition}(Skolemization)
Every Skolem fibration $\mp$ validates the principle:
\[ \forall x \exists y \alpha(i,x,y)\cong \exists f \forall x \alpha (i,x,fx).\]
\end{proposition}
\begin{proof}
Let us consider an element $\alpha \in E_{A_1\x A_2 \x B}$ and a $\coprod$-quantifier-free object $\gamma \in E_{A_1}$.  Hence, for every arrow $\arrow{\pr_1^*(\gamma)}{h}{\coprod_{\angbr{\pr_1}{\pr_2}}(\alpha)}$, there is a unique pair $(g,\ovln{h})$ where $\arrow{A_1\x A_2}{g}{B}$ and $\arrow{\pr_1^*(\gamma)}{\ovln{h}}{\angbr{\pr_1,\pr_2}{g}^*(\alpha)}$. Since $\mB$ has exponents, then we have that $g$ induces a unique arrow $\arrow{A_1}{m}{B^{A_2}}$ such $\angbr{\pr_1,\pr_2}{g}=\angbr{\pr_1,\pr_2}{\ev \angbr{\pr_1}{\pr_2}}\angbr{\pr_1,\pr_2}{m\pr_1}$. Therefore we have an arrow
\[\arrow{\pr_1^*(\gamma)}{\ovln{h}}{ \angbr{\pr_1,\pr_2}{m\pr_1}^*(\angbr{\pr_1,\pr_2}{\ev \angbr{\pr_1}{\pr_2}}(\alpha))}.\]
Since $\mp$ has simply products, $\ovln{h}$ induces a unique arrow 
\[\arrow{\gamma}{\ovln{\ovln{h}}}{\coprod_{\pr_1} \angbr{\pr_1,\pr_2}{m\pr_1}^*(\angbr{\pr_1,\pr_2}{\ev \angbr{\pr_1}{\pr_2}}(\alpha))}.\]
Notice that the following square
\[\xymatrix{ A_1\x A_2 \ar[r]^{\pr_1} \ar[d]_{\angbr{\pr_1,pr_2}{m\pr_1}} & A_1 \ar[d]^{\angbr{id_{A_1}}{m}} \\
A_1\x A_2\x B^{A_2} \ar[r]_{\angbr{\pr_1}{\pr_3}} & A_1\x B^{A_2}}\]
is a pullback, hence by the BCC we have that $\prod_{\pr_1} \angbr{\pr_1,\pr_2}{m\pr_1}^*\cong \angbr{\id_{A_1}}{m}\prod_{\angbr{\pr_1}{\pr_3}}$. Thus, we get that an arrow $f$ induces a unique pair of arrows $(m,\ovln{\ovln{h}})$, but again (since $\mp$ has enough $\coprod$-quantifier-free objects) this pair represents a unique arrow of the fibre $\mE_{A_1}(\gamma, \coprod_{\pr_{3}}\prod_{\angbr{\pr_1}{\pr_1}}(\angbr{\pr_1,\pr_2}{\ev \angbr{\pr_1}{\pr_2}}(\alpha))$, i.e. the fibre 
\[\mE_{A_1\x A_2}(\pr_1^*(\gamma),\coprod_{\angbr{\pr_1}{\pr_2}}(\alpha))\] is isomorphic to \[\mE_{A_1}(\gamma, \coprod_{\pr_{1}}\prod_{\angbr{\pr_1}{\pr_3}}(\angbr{\pr_1,\pr_2}{\ev \angbr{\pr_1}{\pr_2}}(\alpha))\]
and this means exactly that 
\[\prod_{\pr_1}\coprod_{\angbr{\pr_1}{\pr_2}}(\alpha)\cong \coprod_{\pr_{1}}\prod_{\angbr{\pr_1}{\pr_3}}(\angbr{\pr_1,\pr_2}{\ev \angbr{\pr_1}{\pr_2}}(\alpha).\]
The proof for the general case where $\gamma$ is a generic element of the fibre and not a $\coprod$-quantifier-free object, follows by the observation that the arrows $\arrow{\pr^*(\gamma)}{}{\beta}$ are in bijection with those of the form $\arrow{\pr_1^*(\gamma')}{}{\coprod_{\pr_2}\beta}$ for 
appropriate projections, and where $\gamma'$ is the $\coprod$-quantifier-free element which covers $\gamma$.
\end{proof}

Combining Definitions \ref{definition enought-coprod-free-quantifaier}, \ref{definition enought-prod-free-quantifaier} and \ref{definition skolem fibration}, we introduce the notion of a \emph{G\"odel fibration}. The idea is that a G\"odel fibration is a Skolem fibration, such that every formula $\alpha(i)$ is equivalent to a formula in \emph{prenex normal form} with respect to $\mp$, i.e. there exists a formula $\beta(x,y,i)$ free from quantifiers, such that $\alpha(i)\cong \exists x \forall y \beta(x,y,i)$.

\begin{definition}\label{definition PNF-Skolem fibration}
A Skolem fibration $\fibration{\mE}{\mp}{\mB}$ is called a \bemph{G\"odel} fibration if the sub-fibration $\fibration{\mE'}{\mp'}{\mB}$, whose elements are $\coprod$-quantifier-free objects, has enough $\prod$-quantifier-free objects.

\end{definition}
\begin{remark}
Observe that if we consider a G\"odel fibration $\fibration{\mE}{\mp}{\mB}$, an element which is a $\prod$-quantifier-free object in the sub-fibration $\mp'$ could not be $\prod$-quantifier-free object  of the G\"odel fibration. This because in Definition \ref{definition PNF-Skolem fibration} of G\"odel fibration, the universal property of being a $\prod$-quantifier-free object is required to hold only with respect to the $\coprod$-quantifier-free objects of $\mp$.
\end{remark}
The following proposition is an immediate consequence of Definition \ref{definition PNF-Skolem fibration}.
\begin{proposition}[Prenex normal form]
In a G\"odel fibration $\fibration{\mE}{\mp}{\mB}$, for every element $\alpha$ of a fibre $\mE_I$ there exists an element $\beta$ such that
\[ \alpha(i)\cong \exists x\forall y \beta (x,y,i)\]
and $\beta$ is $\prod$-quantifier-free in the sub-fibration $\mp'$ of $\coprod$-quantifier-free objects of $\mp$.
\end{proposition}
\begin{proof}
Let us consider an element $\alpha$ of the fibre $\mE_I$. Since $\mp$ is a G\"odel fibration, hence in particular a Skolem fibration, the fibration $\mp$ has enough $\coprod$-quantifier-free objects, and hence there exists an element $\gamma$ in the fibre $\mE_{I\times X}$ such that $\alpha\cong \coprod_{\pr_I}(\gamma)$. Therefore, since the sub-fibration $\mp'$ has enough $\prod$-quantifier-free objects, there exists a $\prod$-quantifier free element $\beta$ of $\mp'$ in the fibre $\mE_{I\times X\times Y}$ such that $\gamma\cong\prod_{\pr_X}(\beta)$, and hence $\alpha\cong \coprod_{\pr_I}\prod_{\pr_X}(\beta)$.
\end{proof} 
\section{An intrinsic description of the dialectica construction}

\begin{proposition}\label{proposition existence property}
Let $\fibration{\mE}{\mp}{\mB}$ be a fibration, and let us consider the simple coproduct completion $\coprodcomp{\mp}$. Let $I$ be an object of $\mB$ and let $\alpha$ be an object of its fibre $\mE_I$. Then every object of the form $(I,1,\alpha)$ in the fibre $\coprodcomp{\mp}_I$ is $\coprod$-quantifier-free element of $\coprodcomp{\mp}$.



\end{proposition}
\begin{proof}
Let us consider an arrow
$$\arrow{\eta_{\mp}(\alpha)=(I,1,\alpha)}{(f,\phi)}{\coprod_{\pr_I}(I\x A, B,\beta)}$$
where $\arrow{I}{f=\angbr{g_1}{g_2}}{A\x B}$. We are going to prove that $$\arrow{\eta_{\mp}(\alpha)}{(g_2,\phi)}{\clift{\angbr{1_I}{g_1}}(I\x A, B,\beta)}$$ is an arrow of $\coprodcomp{\mp}_{I}$ and that $(f,\phi)=(\clift{\angbr{1_I}{g_1}}\eta_{\beta})(g_2,\phi)$, where $\eta_{\beta}$ is the unit at $(I\x A,B,\beta)$ of the adjunction $\coprod_{\pr_I} \dashv \clift{\pr_I}$.

Moreover, we have to prove that such a choice of arrows $\arrow{I}{g}{A}$ of $\mB$ and $\arrow{\eta_\mp(\alpha)}{\ovln{h}}{\langle 1_I,h \rangle^*(I\times A,B,\beta)}$ of $\coprodcomp{p}_I$ is unique. That is, whenever the equality: $$(f,\phi)=(\clift{\angbr{1_I}{g}}\eta_{\beta})\ovln{h}$$ holds, it is the case that $g=g_1$ and $\ovln{h}=(g_2,\phi)$.

By Remarks \ref{reindexing} and \ref{coprodcomp has simple coprod}, it is the case that $\coprod_{\pr_I}(I\x A, B,\beta)=(I,A\x B,\beta)$, and that $\clift{\angbr{1_I}{g_1}}(I,A\x B,\beta)=(I,B,\clift{\angbr{\pr_I,g_1\pr_I}{\pr_B}}\beta)$, where $\pr_I$ and $\pr_B$ are the projections from $I\x B$. Then: 
$$\arrow{\eta_{\mp}(\alpha)}{(g_2,\phi)}{\clift{\angbr{1_I}{g_1}}(I\x A, B,\beta)=(I,B,\clift{\angbr{\pr_I,g_1\pr_I}{\pr_B}}\beta)}$$
is a morphism of $\coprodcomp{\mp}_I$ since $\arrow{I}{g_2}{B}$ is an arrow of $\mB$ and: 
$$\arrow{\alpha}{\phi}{\clift{\angbr{1_I}{g_2}}\clift{\angbr{\pr_I,g_1\pr_I}{\pr_B}}\beta=\clift{\angbr{1_I,g_1}{g_2}}\beta}=\langle 1_I,f\rangle ^*\beta$$
is a vertical morphism of $\mE_I$. Observe that $\eta_\beta$ is the transpose along the adjunction  $\coprod_{\pi_I}\dashv \pi^*_I$ of the identity arrow of $(I,A\times B,\beta)=\coprod_{\pi_I}(I\times A,B,\beta)$. Hence $\eta_\beta$ is the arrow: $$\arrow{(I\times A,B,\beta)}{(\pi_{A\times B},1_\beta)}{(I\times A,A\times B,\langle \pi_I,\pi_{A\times B} \rangle^*\beta)}$$ and its $\langle 1_I,g_1\rangle$-reindexing is the arrow: $$\arrow{(I,B,\langle\pi_I,g_1\pi_I,\pi_B \rangle^*\beta)}{(\langle g_1\pi_I,\pi_B\rangle,\;1_{\langle\pi_I,g_1\pi_I,\pi_B \rangle^*\beta})}{(I,A\times B,\beta)} $$ whose precomposition by $(g_2,\phi)$ yields indeed the arrow $(f,\phi)$.

\medskip

Let us assume that the equality: \begin{equation} (f,\phi)=(\clift{\angbr{1_I}{g}}\eta_{\beta})\ovln{h} \end{equation} holds for some arrow $\arrow{I}{g}{A}$ of $\mB$ and $\arrow{\eta_\mp(\alpha)}{\ovln{h}=(h_2,\psi)}{\langle 1_I,h \rangle^*(I\times A,B,\beta)}$ of $\coprodcomp{p}_I$. As it is the case that: $$\clift{\angbr{1_I}{g}}\eta_{\beta}=(\langle g\pi_I,\pi_B\rangle,\;1_{\langle\pi_I,g\pi_I,\pi_B \rangle^*\beta})$$ one might compute the right-hand side of the equality (1) and infer the equality: $$(f,\phi)=(\langle g,h_2\rangle,\psi)$$ which implies that $g=g_1$ and $\ovln{h}=(h_2,\psi)=(g_2,\phi)$. \end{proof}

\begin{remark}\label{remark every arrow in the fibres factors in a canonical way}
Let $\fibration{\mE}{\mp}{\mB}$ be a fibration and let $I$ be an object of $\mB$. Let us consider an arrow $\arrow{(I,A,\alpha)}{(f,\phi)}{(I,B,\beta)}$ of $\coprodcomp{\mp}_I$. As $(I,A,\alpha)=\coprod_{\pi_I}(I\times A,1,\alpha)$, we might consider its exponential transpose $\arrow{(I\times A,1,\alpha)}{(1_{I \x A},f,\phi)}{\pi_I^*(I,B,\beta)=(I\times A,B,\langle \pi_I,\pi_B\rangle^*\beta)}$, which is the unique arrow making the following diagram: $$\xymatrix{ \coprod_{\pi_I}(I\x A,1,\alpha)\ar[rr]^{(1_I,f,\phi)} \ar[d]_{\coprod_{\pi_I}(1_{I\times A},f,\phi)}&& (I,B,\beta)\\
\coprod_{\pi_I} \pr_I^*(I,B,\beta)\ar[rru]_{\varepsilon_{(I,B,\beta)}}
}$$ commute. Moreover, as $(I\times A,B,\langle \pi_I,\pi_B \rangle^*\beta)=\coprod_{\pi_{I\times A}}(I\times A\times B,1,\langle \pi_I,\pi_B \rangle^*\beta)$, by Proposition \ref{proposition existence property}, the arrow $(1_{I\times A},f,\phi)$ factors uniquely as the arrow: $$(I\times A,1,\alpha)\xrightarrow{(!,\phi)}\langle 1_{I\times A},f \rangle^*(I\times A\times B,1,\langle\pi_I,\pi_B\rangle^*\beta)=(I\times A,1,\langle \pi_I,f\rangle^*\beta)$$ (which can be uniquely expressed as the image $(\ovln{\mp}\hookrightarrow\coprodcomp{\ovln{\mp}})\phi$ of the arrow $\arrow{\alpha}{\phi}{\langle \pi_I,f\rangle^*\beta}$ of $\ovln{E}_{I\times A}$) followed by the arrow: $$(I\times A,1,\langle \pi_I,f\rangle^*\beta)\xrightarrow{\langle 1_{I\times A},f\rangle^*\eta_{(I\times A\times B,1,\langle \pi_I,\pi_B \rangle^*\beta)}}\coprod_{\pi_{I\times A}}(I\times A\times B,1,\langle \pi_I,\pi_B \rangle^*\beta)$$ which is the $\langle 1_{I\times A},f \rangle$-reindexing of the unit: $$\arrow{(I\times A\times B,1,\langle \pi_I,\pi_B\rangle^*\beta)}{\eta_{(I\times A\times B,1,\langle \pi_I,\pi_B\rangle^*\beta)}}{(I\times A\times B,B,\langle \pi_I,\pi_B\rangle^*\beta)}$$ of the adjunction $\coprod_{\pi_{I\times A}}\dashv \pi_{I\times A}^*$.
\end{remark}
Notice that in Proposition \ref{proposition existence property} the elements of the form $(I,1,\alpha)$ represent propositions which are free from the existential quantifier.

\begin{remark}
The analogous of Remark \ref{remark every arrow in the fibres factors in a canonical way} can be proved for a fibration having enough $\coprod$-quantifier-free objects. In other words, in this kind of fibrations the arrows of the fibres are completely described by arrows between quantifier-free objects, unit and counit of adjunctions given by coproducts.
\end{remark}
\begin{proposition}
Whenever $\fibration{\mE}{\mp}{\mB}$ is a fibration, then the $\coprod$-quantifier-free objects of $\coprodcomp{\mp}$ are up to isomorphism the elements of the form $(I,1,\alpha)$. In particular, since every object $(I,B,\beta)$ of $\coprodcomp{\mp}$ satisfies: $$(I,B,\beta)\cong\coprod_{\pi_I}(I\times B,1,\beta)$$ where $\pi_I$ is the projection $I\times B \to I$, it is the case that $\coprodcomp{\mp}$ has enough $\coprod$-quantifier-free objects.

\proof
Whenever $f$ is an arrow $A \to I$ of $\mB$, it is the case that the $f$-reindexing of $(I,1,\alpha)$ is the triple $(A,1,f^*\alpha)$, which is still a \textit{quantifier free formula}, that is, its second component is terminal in $\mB$. Hence, the $\coprod$-quantifier-freeness of $(I,1,\alpha)$ follows from Proposition \ref{proposition existence property}.

Viceversa, let us assume that the triple $(I,A,\alpha)$ is $\coprod$-quantifier-free and let us consider its identity arrow $(I,A,\alpha)\to(I,A,\alpha)=\coprod_{\pi_I}(I\times A,1,\alpha)$. By $\coprod$-quantifier-freeness, there are an arrow $\arrow{I}{g}{A}$ of $\mB$ and an arrow: $$\arrow{(I,A,\alpha)}{(\arrow{I\times A}{!}{1},\;\arrow{\alpha}{\phi}{\pi_I^*\langle 1_I,g \rangle^*\alpha=\langle\pi_I,g\pi_I \rangle^*\alpha})}{\langle 1_I,g \rangle^*(I\times A,1,\alpha)=(I,1,\langle1_I,g \rangle^*\alpha)}$$ of $\coprodcomp{\mp}_I$ such that the identity arrow $(\pi_A,1_\alpha)$ of $(I,A,\alpha)$ equals the composition: $$\Big(\;\arrow{(I,A,\alpha)}{(\arrow{I\times A}{!}{1},\;\phi)}{(I,1,\langle1_I,g \rangle^*\alpha)}\xrightarrow{(g,1_{\langle 1_I,g\rangle^*\alpha})}(I,A,\alpha)\;\Big)$$ where the couple $(g,1_{\langle 1_I,g\rangle^*\alpha})$ is nothing but the $\langle 1_I,g\rangle$-reindexing of the unit at $(I\times A,1,\alpha)$ of the adjunction $\coprod_{\pi_I}\dashv \pi_I^*$. We infer by this arrow equality that it needs to be the case that $(\arrow{I\times A}{\pi_A}{A})=(\arrow{I\times A}{\pi_I}{I}\xrightarrow{g}A)$ and that: $$(\;\arrow{\alpha}{\phi}{\langle \pi_I,g\pi_I\rangle^*\alpha=\langle \pi_I,\pi_A\rangle^*\alpha=\alpha}\xrightarrow{1_\alpha}\alpha\;)=1_\alpha$$ which means that $\phi=1_\alpha$. Finally we observe that the composition: $$(I,1,\langle 1_I,g\rangle^*\alpha)\xrightarrow{(g,1_{\langle1_I,g\rangle^*\alpha})}(I,A,\alpha)\xrightarrow{(\arrow{I\times A}{!}{1},\;\phi=1_\alpha)}(I,1,\langle 1_I,g\rangle^*\alpha)$$ equals the identity arrow $(\arrow{I\times 1}{!}{1},1_{\langle 1_I,g\rangle^*\alpha})$. This concludes that the couple $(\arrow{I\times A}{!}{1},1_\alpha)$ is actually an isomorphism $(I,A,\alpha)\cong (I,1,\langle 1_I,g\rangle^*\alpha)$.
\endproof
\end{proposition}

Now we have all the instrument to give an \emph{internal} description of the free-algebras for the pseudo-monad which adds the simple coproducts to a given fibration.
\begin{theorem}\label{proposition free-algebras coprod}
A fibration $\fibration{\mE}{\mp}{\mB}$ with simple coproducts is an instance of simple coproduct completion if and only if it has enough $\coprod$-quantifier-free objects.
\end{theorem}
\begin{proof}
We define $\fibration{\ovln{\mE}}{\ovln{\mp}}{\mB}$ the full-subfibration of $\fibration{\mE}{\mp}{\mB}$ such that the objects of $\ovln{\mE}$ are the $\coprod$-free-quantifers. By the universal property of the inclusion morphism $\ovln{\mp}\hookrightarrow\coprodcomp{\ovln{\mp}}$, there is unique a morphism of fibrations with simple coproducts $\fibration {\coprodcomp{\ovln{\mp}}}{F}{\mp}$ commuting with the inclusion morphisms $\ovln{\mp}\overset{\eta_{\ovln{\mp}}}{\hookrightarrow}\coprodcomp{\ovln{\mp}}$ and $\ovln{\mp}\hookrightarrow\mp$. We claim that $F$ is an equivalence of fibrations. At first we observe its essential surjectivity and secondly its full faithfulness. From now on, whenever $\pi$ is a projection in $\mB$, we indicate as $\coprod_\pi$ the left adjoint to the $\pi$-weakening w.r.t. $\coprodcomp{\ovln{p}}$ and as $\sum_\pi$ the one w.r.t. $\mp$. Observe that: $$F(I,1,\gamma)=F\big(\ovln{\mp}\overset{\eta_{\ovln{\mp}}}{\hookrightarrow}\coprodcomp{\ovln{\mp}}\big)\gamma = (\ovln{\mp}\hookrightarrow \mp)\gamma = \gamma$$ for every $I$ in $\mB$ and every $\gamma$ in $\ovln{E}_I$.

\medskip

\textit{Essential surjectivity.} Let $\alpha$ be an object of $\mE$ and let $I$ be the object $\mp\alpha$ of $\mB$. Since $\mp$ has enough $\coprod$-quantifier-free objects, there are $J$ in $\mB$ and $\beta$ in $\mE_{I \times J}$ such that $\sum_{\pi_I}\beta\cong \alpha$. Since $F$ preserves simple coproducts, it is the case that: $$\begin{aligned}F(I,J,\beta)&=F\coprod_{\pi_I}(I\times J,1,\beta)=\sum_{\pi_I}F(I\times J,1,\beta)=\sum_{\pi_I}\beta\end{aligned}$$ and we are done. Observe that $(I,J,\beta)$ is an object of $\mE_I$, hence the functor $\mE_I\to \mE_I'$ induced by $F$ is essentially surjective as well.

\medskip

\textit{Full faithfulness.} We use Lemma \ref{fully faithful} and prove that, for a given object $I$ of $\mB$, the functor $F$ gives rise to an equivalence $\mE_I \to \mE'_I$. As the essential surjectivity of $F\upharpoonright_{\mE_I}\colon\mE_I \to \mE'_I$ follows by the previous part, we only need to observe its full faithfulness.

By Remark \ref{remark every arrow in the fibres factors in a canonical way} we write a given arrow $\arrow{(I,A,\alpha)}{(f,\phi)}{(I,B,\beta)}$ of $E_I$ as the composition: $$\varepsilon_{(I,B,\beta)}\Big(\coprod_{\pi_I}\langle 1_{I\times A},f\rangle^*\eta_{(I\times A\times B,1,\langle \pi_I,\pi_B \rangle^*\beta)}\Big)\Big(\coprod_{\pi_I}(\ovln{\mp}\hookrightarrow\coprodcomp{\ovln{\mp}})\phi\Big)$$ and this factorisation is unique, because of the uniqueness of adjoint transposition, because of the uniqueness-part of Proposition \ref{proposition existence property} and because of faithfulness of the functor $\ovln{\mp}\hookrightarrow\coprodcomp{\ovln{\mp}}$. As $F$ is forced to preserve simple coproducts and commutes with the inclusion morphisms $\ovln{\mp}\overset{\eta_{\ovln{\mp}}}{\hookrightarrow}\coprodcomp{\ovln{\mp}}$ and $\ovln{\mp}\hookrightarrow\mp$, the arrow $F(f,\phi)$ equals the composition: $$\varepsilon_{(\sum_{\pi_I}\beta)}\Big(\sum_{\pi_I}\langle 1_{I\times A},f\rangle^*\eta_{\langle \pi_I,\pi_B \rangle^*\beta}\Big)\Big(\sum_{\pi_I}\phi\Big)$$ which is indeed an arrow $\sum_{\pi_I}\alpha \to \sum_{\pi_I}\beta$. Observe that, analogously, every arrow $\sum_{\pi_I}\alpha \to \sum_{\pi_I}\beta$ of $\mE'_I$ can be uniquely factored as such a composition, again by the existence and the uniqueness of the adjoint transposition, by Definition \ref{definition P-atomic} (remind that $\mp$ is assumed to have $\coprod$-quantifier-free objects) and by full faithfulness of $\ovln{\mp}\hookrightarrow \mp$. Hence the function: $$\mE_I((I, A,\alpha),(I,B,\beta))\to \mE'_I\Big(\sum_{\pi_I}\alpha,\sum_{\pi_I}\beta\Big)$$ induced by $F\upharpoonright_{\mE_I}$ is bijective, i.e. $F\upharpoonright_{\mE_I}$ is fully faithful.
\end{proof}
Notice that the characterization of Theorem \ref{proposition free-algebras coprod} can be obtained also for the simple product completion thanks to the equivalence $\prodcomp{\mp}\cong \coprodcomp{\mp^{\op}}^{\op}$, see Proposition \ref{proposition coproc is dual to prod}.
\begin{theorem}\label{proposition free-algebras prod}
A fibration $\fibration{\mE}{\mp}{\mB}$ with simple products is an instance of simple product completion if and only if it has enough- $\prod$-quantifier-free objects.
\end{theorem}
\begin{proof}
It follows by Lemma \ref{opfibration}, Theorem \ref{proposition free-algebras coprod} and Proposition \ref{proposition coproc is dual to prod}.
\end{proof}

Combining Theorem \ref{proposition free-algebras coprod} and Theorem \ref{proposition free-algebras prod} we can prove the following theorem, which allows us to recognize if an arbitrary fibration $\mp$ is an instance of the Dialectica construction or not, and if it is, we can construct the fibration $\mp'$ such that $\dial{\mp'}\cong \mp$.

\begin{theorem}\label{theorem godel if and only if Dialectica}
Let $\fibration{\mE}{\mp}{\mB}$ be a fibration with products, coproducts and such that $\mB$ is cartesian closed. Then there exists a fibration $\mp'$ such that for $\dial{\mp'}\cong \mp$ if and only if $\mp$ is a G\"odel fibration.

\end{theorem}

\begin{remark}
Notice that from a categorical perspective Theorem \ref{theorem godel if and only if Dialectica} provides a characterization of the free-algebras of the pseudo-monad $\dial{-}$.
\end{remark}

\section{On fibred (weak) finite (co)products}
The categorical properties of Dialectica categories were first studied by de Paiva \cite{dePaiva1989Dialectica}, and then by Hyland \cite{HYLAND200243} and Biering \cite{Biering_Dialecticainterpretations}. 
Observe that all of these presentations of the Dialectica construction correspond to the fibre $\dial{\mp}_1$ (associated to an 
appropriate fibration $\mp$) of the Dialectica monad introduced in \cite{hofstra2011}.  
The main goal of this section is to generalize the previous results to the total category $\dial{\mp}$. For this purpose, we take advantage of the decomposition $\dial{-}\cong \coprodcomp{ \prodcomp{-}} $ of the Dialectica monad into the two (dual) constructions, and we reduce the problem to the study of just one of these two.

We first provide the minimal hypotheses that a fibration $\mp$ with fibred weak finite products and coproducts is required to satisfy so that the Dialectica monad preserves these structures. These hypotheses happen to be preserved as well. Then we show that the total category itself $\dial{\mp}$ has finite weak products and coproducts.

\begin{remark}
 Most of the results contained in the current section arise as generalisations of the ones contained in \cite{trottapasquale2020}.
 Let us consider for instance the following statement: \begin{quote}
    Let $\mathcal{C}$ be a category with finite products and let $\fibration{\mathcal{C}^{\op}}{P}{\infsl}$ be a doctrine, where $\infsl$ is the category of inf-semilattices, i.e. finitely complete posets, and inf-preserving maps, i.e. finite limit preserving functors. Then the existential doctrine $\fibration{\mathcal{C}^{\op}}{P^{\existential}}{\pos}$ obtained from the exact completion of $P$ factors through the forgetful functor $\infsl \hookrightarrow \pos$, i.e. $P^{\existential}$ is still a doctrine $\mathcal{C}^{\op}\longrightarrow\infsl$.
\end{quote} whose proof is contained in \cite{trotta2020}. It basically states that the existential completion of a doctrine preserves the existence of finite meets in its \textit{power-sets} (see \cite{trottapasquale2020} and \cite{trotta2020} for more details).
As long as one becomes interested in \textit{proof-relevance}, the notion of doctrine can be generalised, depending on the richness of its power-sets, that is, on its codomain category. For instance, if one is interested in proof-relevant doctrines over a cartesian category $\mathcal{C}$ and whose codomain in principle is $\infsl$, then the right generalised notion of such a doctrine is the one of a functor $\mathcal{C}^{op} \longrightarrow \cartcat$, where $\cartcat$ is the category of categories with finite products and finite product preserving functors. As we prove in the current section, the previous statement generalises into asserting that the proof-relevant existential completion of such a generalised doctrine is still a doctrine whose codomain is $\cartcat$.
However, here we state and prove such a result, as well as the other ones, by using the language of cloven and split fibrations in place of the equivalent one of doctrines. In fact, inside this framework, ordinary doctrines of codomain $\pos$ ($\infsl$ respectively) correspond to poset-fibrations (inf-semilattice-fibrations respectively), i.e. fibrations whose fibers are posets (inf-semilattices respectively). Analogously, generalised doctrines whose generalises power-sets enjoy some categorical property correspond to ordinary cloven and split fibrations whose fibers enjoy the same categorical property.
\end{remark}

\begin{proposition}\label{fibredproductpreservation}
The simple coproduct completion preserves fibred (weak) finite products, i.e. whenever a fibration $\fibration{\mE}{\mp}{\mB}$ has fibred (weak) finite products, then its simple coproduct completion $\fibration{\coprodcomp{\mp}}{\coprodcomp{\mp}}{\mB}$ has fibred (weak) finite products as well. Dually, the simple product completion preserves fibred (weak) finite coproducts.
\end{proposition}
\begin{proof}
We only prove the first part of the statement, as the second follows by Lemma \ref{opfibration}, Proposition \ref{proposition coproc is dual to prod} and Proposition \ref{fibredproductpreservation}. Let $I$ be an object of $\mB$ and let us consider two objects $(I,X,\alpha)$ and $(I,Y,\beta)$ of the fibre $\coprodcomp{\mp}_{I}$. Then the triple:
$$(I,X\x Y, \clift{\angbr{\pr_I}{\pr_X}}\alpha\x \clift{\angbr{\pr_I}{\pr_Y}}\beta)$$
is a (weak) product $(I,X,\alpha)\x (I,Y,\beta)$ in $\coprodcomp{\mp}_I$ together with the projections: $$(\pr_X,\pr_{\clift{\angbr{\pr_I}{\pr_X}}\alpha}) \text{ and }(\pr_Y,\pr_{\clift{\angbr{\pr_I}{\pr_Y}}\beta})$$ where the $\pr_I,\pr_X,\pr_Y$ are the projections from $I\x X\x Y$, and $\pr_{\clift{\angbr{\pr_I}{\pr_Y}}\alpha}$, $\pr_{\clift{\angbr{\pr_I}{\pr_Y}}\beta}$ are the projections from $\clift{\angbr{\pr_I}{\pr_X}}\alpha\x \clift{\angbr{\pr_I}{\pr_Y}}\beta$. In fact, whenever $(f,\phi)$ and $(g,\psi)$ are arrows: $$(I,Z,\gamma)\to(I,X,\alpha)\text{ and } (I,Z,\gamma) \to (I,Y,\beta)$$ then the couple: $$\big(\;\langle f,g \rangle,\;\arrow{\gamma}{\langle \phi,\psi\rangle}{\langle \pi_I,f\rangle^*\alpha\times \langle\pi_I,g \rangle^*\beta=\langle\pi_I, \langle f,g\rangle\rangle^*(\langle\pi_I,\pi_X\rangle^*\alpha \times \langle\pi_I,\pi_Y\rangle^*\beta \rangle)\;}\big)$$ is an arrow $(I,Z,\gamma)\to (I,X\times Y,\langle\pi_I,\pi_X\rangle^*\alpha \times \langle\pi_I,\pi_Y\rangle^*\beta \rangle)$ making the usual triangles commute and, if $\clift{\angbr{\pr_I}{\pr_X}}\alpha\x \clift{\angbr{\pr_I}{\pr_Y}}\beta$ is strict, it is the unique one by the universal properties of $X\times Y$ and $\clift{\angbr{\pr_I}{\pr_X}}\alpha\x \clift{\angbr{\pr_I}{\pr_Y}}\beta$ in $\mB$ and $\mE_{I\times X \times Y}$ respectively. Moreover observe that, whenever $\arrow{J}{h}{I}$ is an arrow of $\mB$, it is the case that the $h$-reindexing of $(\;I,X\times Y,\langle\pi_I,\pi_X\rangle^*\alpha \times \langle\pi_I,\pi_Y\rangle^*\beta \rangle)$ is the object: $$\big(J,X\times Y,\langle h\pi_J,\pi_X,\pi_Y\rangle^*(\langle\pi_I,\pi_X\rangle^*\alpha \times \langle\pi_I,\pi_Y\rangle^*\beta \rangle)=\langle h\pi_J,\pi_X \rangle^*\alpha \times \langle h\pi_J,\pi_Y \rangle^*\beta\;\big)$$ which is the (weak) product of the objects $(J,X,\langle h\pi_J,\pi_X \rangle^*\alpha)$ and $(J,Y,\langle h \pi_J,\pi_Y\rangle^*\beta)$ in $\coprodcomp{\mp}_J$.

Finally, a (weak) terminal object of the fibre $\coprodcomp{\mp}_I$ is $(I,1,1_I)$ whenever $1$ is terminal in $\mB$ and its $h$-reindexing is $(J,1,1_J)$.
\end{proof}

By Proposition \ref{fibredproductpreservation} we cannot conclude that the Dialectica fibration $\dial{\mp}$ has either fibred (weak) finite coproducts or products, since $\dial{-}$ is given by the composition of $\coprodcomp{-}$ and $\prodcomp{-}$, and the previous proposition just shows that these two completions preserve different structure.  Therefore we need to provide the right assumptions such that the simple coproduct completion preserves fibred (weak) finite coproducts and, dually, the assumptions such that the simple product completion preserves fibred (weak) finite products.

\begin{proposition}\label{falso!}
Let $\mB$ be a distributive category and let $\fibration{\mE}{\mp}{\mB}$ be a fibration with fibred (weak) finite coproducts such that the cartesian liftings w.r.t. $\mp$ of the injections $\arrow{A}{j_A}{A+B}$ have left adjoints $\coprod_{j_A}$ which satisfy BCC for pullbacks (when they exist) of injections which are injections themselves. Then the fibration $\coprodcomp{\mp}$ has fibred (weak) finite coproducts.

\end{proposition}
\begin{proof}
Let $I$ be an object of $\mB$. Then, whenever $0$ is initial in $\mB$, the triple $(I,0,0_I)$ is a (weak) initial object of the fibre $\coprodcomp{\mp}_I$ and, whenever $\arrow{J}{h}{I}$ is an arrow of $\mB$, its $h$-reindexing is $(J,0,0_J)$. Moreover, whenever $(I,X,\alpha)$ and $(I,Y,\beta)$ are objects of $\coprodcomp{\mp}_I$, then the object: $$\Big(I, X+Y, \clift{\big(\theta^{-1}\big)}\bigg(\coprod_{j_{I\x X}}\alpha+ \coprod_{j_{I\x Y}}\beta\bigg)=:\chi \Big)$$ where $\arrow{(I\x X)+(I\x Y)}{\theta}{I\x (X+Y)}$ is the canonical isomorphism, is a (weak) coproduct $(I,X,\alpha)+(I,Y,\beta)$ in $\coprodcomp{\mp}_I$ together with the injections:
$$(j_X\pr_X=\pi_{X+Y}\theta j_{I\times X},\;\ovln{j}_{\alpha}\eta_{\alpha})\text{ and }(j_Y\pr_Y=\pi_{X+Y}\theta j_{I\times Y},\;\ovln{j}_{\beta}\eta_{\beta})$$ where $\ovln{j}_\alpha\eta_\alpha$ and $\ovln{j}_\beta\eta_\beta$ are the compositions: $$\arrow{\alpha}{\eta_\alpha}{j^*_{I\times X}\coprod_{j_{I\times X}}\alpha\overset{\overline{j}_\alpha}}{\hookrightarrow}(\theta j_{I\times X})^*\chi=\langle \pi_I, j_X\pi_X \rangle^*\chi$$ and: $$\arrow{\beta}{\eta_\beta}{j^*_{I\times Y}\coprod_{j_{I\times Y}}\beta\overset{\overline{j}_\beta}}{\hookrightarrow}(\theta j_{I\times Y})^*\chi=\langle \pi_I, j_Y\pi_Y \rangle^*\chi.$$ In fact, whenever: $$(\;f,\;\arrow{\alpha}{\phi}{\langle\pi_I, f\rangle^*\gamma=j_{I\times X}^*\langle \pi_I\theta,[f,g]\rangle^*\gamma}\;) \text{ and } (\;g,\;\arrow{\alpha}{\psi}{\langle\pi_I, g\rangle^*\gamma=j_{I\times Y}^*\langle \pi_I\theta,[f,g]\rangle^*\gamma}\;)$$ are arrows $(I,X,\alpha)\to(I,Z,\gamma) \text{ and }(I,Y,\beta)\to(I,Z,\gamma)$, denoted as $\tilde{\phi}$ and $\tilde{\psi}$ the transposed of $\phi$ and $\psi$ respectively along the adjunctions $\coprod_{I\times X}\dashv j_{I\times X}^*$ and $\coprod_{I\times Y}\dashv j_{I\times Y}^*$ respectively, then the couple: $$\big(\; [f,g]\theta^{-1},\; \chi \xrightarrow{(\theta^{-1})^*\big[\tilde{\phi},\tilde{\psi}\big]}\langle \pi_I,\langle f,g \rangle\theta^{-1}\rangle^*\gamma \;\big)$$ is an arrow $(I,X\times Y,\chi) \to (I,Z,\gamma)$ making the usual triangles commute. If $\chi$ is strict, then it is the unique one by the universal properties of $X + Y$ and $\chi$ in $\mB$ and $\mE_{I\times (X+Y)}$ respectively and because $\theta$, $\eta_\alpha$ and $\eta_\beta$ are isomorphisms. Observe indeed that the unit of the adjunction $\coprod_{I\times X}\dashv j_{I\times X}^*$ needs to be a natural isomorphism by BCC and because the injection $j_{I\times X}$ is a monomorphism, being $\beta$ distributive - the same holds for the unit of $\coprod_{I\times Y}\dashv j_{I\times Y}^*$.

Whenever $\arrow{J}{h}{I}$ is an arrow of $\mB$, the $h$-reindexing of $(I,0,0_I)$ is $(J,0,0_J)$ and the $h$-reindexing of $(I,X+Y,\chi)$ is the triple $(J,X+Y,\langle h\pi_J,\pi_{X+Y} \rangle^*\chi)$ and the following equalities hold: $$\begin{aligned} \langle h\pi_J,\pi_{X+Y} \rangle^*\chi&=(\theta^{-1})^*[\langle h\pi_J,\pi_X\rangle,\langle h\pi_J,\pi_Y \rangle]^*\bigg(\coprod_{j_{I\x X}}\alpha+ \coprod_{j_{I\x Y}}\beta\bigg) \\ &= (\theta^{-1})^*\bigg(\coprod_{j_{J\x X}}\langle h\pi_J,\pi_X\rangle^*\alpha+ \coprod_{j_{J\x Y}}\langle h\pi_J,\pi_Y\rangle^*\beta\bigg)
\end{aligned}$$ by BCC and being the squares: $$\xymatrix{J\times X \ar[d]_{\langle h\pi_J,\pi_X \rangle} \ar[r]^>>>>>{j_{J\times X}} & J\times X + J \times Y \ar[d] && J\times Y \ar[d]_{\langle h\pi_J,\pi_Y \rangle} \ar[r]^>>>>>{j_{J\times Y}} & J\times X + J \times Y \ar[d] \\ I \times X \ar[r]^>>>>>{j_{I\times X}} & I \times X + I \times Y && I \times Y \ar[r]^>>>>>{j_{I\times Y}} & I \times X + I \times Y}$$
pullbacks whose right-hand vertical arrow is the morphism $[\langle h\pi_J,\pi_X \rangle,\langle h\pi_J,\pi_Y \rangle]$. Hence the triple $(J,X+Y,\langle h\pi_J,\pi_{X+Y} \rangle^*\chi)$ is the (weak) coproduct of the objects: $$(J,X,\langle h\pi_J,\pi_X \rangle^*\alpha)\text{ and }(J,Y,\langle h \pi_J,\pi_Y\rangle^*\beta)$$ in $\coprodcomp{\mp}_J$. We conclude that the $h$-reindexing w.r.t. $\coprodcomp{\mp}$ preserves (weak) finite coproducts.
\end{proof}

The proof-irrelevant simple coproduct completion of poset-fibrations preserves the existence of left adjoints to the cartesian lifting whenever the base category has points. That is, whenever a poset-fibration over a category with points has left adjoints to its cartesian liftings of the injections, then its proof-irrelevant simple coproduct completion keeps on having left adjoints to its cartesian liftings of the injections (see \cite{trottapasquale2020} for a proof of this in the equivalent language of doctrines). However, this property is not enjoyed by the proof-relevant simple coproduct completion in general, as the example  below shows.

Thus a dual version of Proposition \ref{falso!} together with Proposition \ref{fibredproductpreservation} would allow us to conclude that $\dial{\mp}$ has fibred (weak) finite products, but the failure of the preservation of adjoints along injections prevents us from concluding that $\dial{\mp}$ has fibred (weak) finite coproduct. 

\begin{example}\label{esempio stupido}
Let us assume that there is a choice of an element $a$ of $A$ for every nonempty set $A$ and let us consider the subobject-fibration over $\set$. This is the fibration $\mp \colon \alexandre{\set}\to \set$ given by the projection on the first component, where $\alexandre{\set}$ is the Grothendieck category of $\set$, whose objects are the pairs $(A,\alpha)$, where $A$ is a set and $\alpha$ is a subset of $A$, and whose arrows $(A,\alpha)\to (B,\beta)$ are the arrows $\arrow{A}{f}{B}$ such that $\alpha\subseteq f^*\beta$. Observe that $\set$ has points and that, whenever $\arrow{A}{j_A}{A+B}$ is an injection in $\set$, then the reindexing $j_A^*$ has a left adjoing $\coprod_{j_A}$ given by the assignment $(A,\alpha\subseteq A) \mapsto (A+B,\alpha\subseteq A+B)$ and BCC 
for pullbacks is preserved. However, this is not the case with respect to the fibration $\coprodcomp{\mp}\to\set$.

Let us consider the objects $(A,C,\gamma)$ and $(A+B,D,\delta)$ of $\coprodcomp{\mp}_A$ and $\coprodcomp{\mp}_{A+B}$ respectively. Then the homset $\coprodcomp{\mp}_{A+B}(\coprod_{j_A}(A,C,\gamma),(A+B,D,\delta))$ is the set of arrows $\arrow{(A+B)\times C}{g}{D}$ such that $\gamma \subset \langle\pi_{A+B}, g\rangle^*\delta$, while the homset $\coprodcomp{\mp}_{A}((A,C,\gamma),j_A^*(A+B,D,\delta))$ is the set of arrows $A \times C\xrightarrow{h} D$ such that $\gamma \subset \langle \pi_A,h\rangle^* j_{A\times D}^*\delta=\langle \pi_A,h\rangle^*((A\times D)\cap \delta)$. Then the usual assignments: $$\begin{aligned} g \overset{\sharp}{\mapsto} \big(\;\arrow{A\times C}{gj_{A\times C}}{D}\;\big) \text{ and } h \overset{\flat}{\mapsto} \big(\;{A\times C + B \times C}&\overset{[h,d]}{\to}{D}\;\;\;\;\;\;\;\big) \\ (a,c)&\;\mapsto\; h(a,c) \\ (b,c)&\; \mapsto\; d \end{aligned}$$ (here $d$ is the element of $D$ in our choice) in general do not define a bijection of these two homsets, but a retraction only: it is indeed the case that the latter is a retract of the former, i.e. $\sharp$ is a retraction of $\flat$.

The pair of functors $\coprod_{j_A} \radjunction j_A^*$ is an example of the notion of \textit{right-weak adjunction}, see Appendix \ref{weak adjunctions}.

Similar reasoning can be done for the preservation of the right adjoints to the reindexing with respect to $\mp$ along the injections, which we know exist. Again, if we denote by $\prod_{j_A}$ the usual candidate for the right adjoint of the $j_A$-reindexing with respect to $\coprodcomp{\mp}$, then we realise that, between the homsets: $$\coprodcomp{\mp}_{A}(j_A^*(A+B,D,\delta),(A,C,\gamma))\text{ and } \coprodcomp{\mp}_{A+B}((A+B,D,\delta),\prod_{j_A}(A,C,\gamma))$$ in general there is not a natural bijection, but just a natural retraction. That is, the former is a retract of the latter, and this retraction is natural. This example motivates the notion of left-weak adjunction (see \ref{left-weak adjunction} in the Appendix), a dual notion to Definition \ref{right-weak adjunction}.
\end{example}

The notions of right-weak and left-weak adjunction appear in a similar context in \cite{moss2018}. It turns out that the property of existence of right-weakly left adjoints to the cartesian liftings of the injections is enough to make the simple coproduct completion of fibrations preserve the existence of fibred weak finite coproducts. As usual, the dual result holds for the simple product completion together with the fibred weak finite products. Hence, we introduce the notion of an \emph{extendable fibration}, which explains those categorical structures which are compatible, i.e. preserved, by the Dialectica interpretation.

\begin{definition}\label{definition godel fibration}
A cloven, split fibration $\fibration{\mE}{\mp}{\mB}$ is said to be an \bemph{extendable fibration} if it satisfies the following properties:
\begin{enumerate}
    \item[1)]  the category $\mB$ is distributive and has \textit{points}, i.e. there is an arrow $\arrow{1}{a}{A}$ for every non-initial object $A$ of $\mB$;
    \item[2)] the fibration $\mp$ has fibred weak finite products and fibred weak finite coproducts;
    \item[3)] the cartesian liftings with respect to $\mp$ of the injections $\arrow{A}{j_A}{A+B}$ have right-weakly left adjoints $\coprod_{j_A}$ which satisfy BCC for pullbacks (when they exist) of injections which are injections themselves;
    \item[4)] the cartesian liftings with respect to $\mp$ of the injections $\arrow{A}{j_A}{A+B}$ have left-weakly right adjoints $\prod_{j_A}$ which satisfy BCC for pullbacks (when they exist) of injections which are injections themselves.
\end{enumerate}
\end{definition}

\begin{remark}
If $\fibration{\mE}{\mp}{\mB}$ is a extendable fibration the total category $\mE$ has weak products, but it does not have weak coproducts, in general. The main reason is that to construct the coproducts in $\mE$ we need to have left adjoints along all reindexing functors, i.e. $\mp$ has to be a bifibration in general, see \cite{Jacobs1999}.
\end{remark}


\begin{example}
 Consider the initial cartesian closed category $\mathbf{T}$ with a natural number object and coproducts.  This category may be built, for example,  using the syntax of System T (possibly adding sums). Recall from \cite{maietti2010} that one can define a \emph{predicate} of $\mathbf{T}$ as a morphism $\arrow{N}{P}{N}$ such that $P\ast P= P$, where $\arrow{N\x N}{\ast}{N}$ is the multiplication of natural numbers. In essence, a predicate is a morphism with values 0 or 1. In particular, the collection of predicates forms a boolean algebra, see \cite{maietti2010}, with  bounded  existential  and  universal  quantification  where  the  order  is  defined as follows:$P\leq Q \mbox{ if and only if } P\dot-Q=0$ where $\dot-$ denotes truncated  subtraction. Therefore, we can consider the functor $\skoledoc$, where $\mathbin{Pred}(N)$ is the poset of predicates, and $\mathbin{Pred}_h(P)=Ph$ for every arrow $h$ of $\mathbf{T}$ and $P\in \mathbin{Pred}(N)$.
Notice that, if we consider a predicate $P\in \mathbin{Pred}(N)$, there are two natural ways to extend this predicate to a predicate of $\mathbin{Pred}(N+ N)$.
The idea is the following: consider the particular case a function $\arrow{\mathbb{ N}}{P}{\mathbb{N}}$  such that $P$ has values $0$ or $1$. Then we can define two functions
from $\arrow{\mathbb{N} + \mathbb{N}}{}{\{0,1\}}$: $P_1$ such that $P_1(a,1)=P(a)$ and $P_1(a,0)=0$ and $P_2$ such that $P_2(a,1)=P(a)$ and $P_2(a,0)=1$. Moreover, one can see that if $\mathbin{Pred}_{j_1}(G)=G j_1\leq P$, where $\arrow{\mathbb{N}+\mathbb{N}}{G}{\{0,1\}}$ is a predicate, then $G\leq P_2$, and $P\leq \mathbin{Pred}_{j_1}(G)$ implies $P_1\leq G$. 
 Formally, these assignments are exactly the left and right adjoints to the functor $\mathbin{Pred}_{j_1}$, and more generally, the existence of left and right adjoints along injections can be proved for the arbitrary case $\skoledoc$. Hence the fibration $ \arrow{\int \mathbin{Pred}}{}{\mathbf{T}}$ corresponding to the functor $\skoledoc$ is an extendable fibration.
\end{example}

\begin{example}\label{esempio stupido 2}
The sub-object fibration $\mp \colon \alexandre{\set}\to \set$ on $\set$ is an extendable fibration.
\end{example}

\begin{example}\label{generalization esempio stupido}
The sub-object fibration over a lextensive category, see  \cite{CARBONI1993}, with points is an extendable fibration. This is the fibration $\alexandre{\mathcal{C}}\to \mathcal{C}$,where $\alexandre{\mathcal{C}}$ is the Grothendieck category over $\mathcal{C}$. Let $A,B$ be objects of $\mathcal{C}$ and let $j_A$ be the monic injection $A \to A+B$. Observe that the reindexing along $j_A$ has a left adjoint $\coprod_{j_A}$ which acts as the post-composition by $j_A$: whenever $(A,s)$ is an object of $(\alexandre{\mathcal{C}})_A$, i.e. $s$ is a subobject of $A$, it is the case that $\coprod_{j_A}(A,s)=(A+B,j_A s)$.

Moreover, it has a right adjoint $\prod_{j_A}$, sending an object $(A,\arrow{S}{s}{A})$ of $(\alexandre{\mathcal{C}})_A$ to the couple: $$(A+B,\arrow{S + B}{s + 1_B}{A + B}).$$ Observe that $s+1_B$ is indeed a mono: if $a$ be an arrow $X \to S + B$ then $X$ has a structure of coproduct together with the injections $\arrow{X_S}{a^*i_S}{X}$ and $\arrow{X_B}{a^*i_B}{X}$ obtained by pulling back along $a$ the injections: $$\arrow{S}{i_S}{S+B}\text{ and }\arrow{B}{i_B}{S+B}$$ respectively. Let us denote as $a_1$ and $a_2$ the unique arrows $X_S \to S$ and $X_B \to B$ respectively such that $a =a_1 + a_2$. Observe that the injections $\arrow{S}{i_S}{S+B}$ and $\arrow{B}{i_B}{S+B}$ are the pullbacks along $s + 1_B$ of the arrows $j_A$ and $j_B$, hence $a^*i_S$ and $a^*i_B$ are the pullbacks of $j_A$ and $j_B$ alons $(s+1_B)a$. This implies that, whenever $b$ is another arrow $X \to S + B$ such that $(s+1_B)a=(s+1_B)b$ then, by appliying the same procedure to $b$, we obtain the same coproduct structure $(a^*i_S,a^*i_B)$ over $X$. In particular $j_A s a_1 = (s+1_B)a(a^*i_S)=(s+1_B)b(a^*i_S)=j_A s b_1$, which implies that $a_1=b_1$, as $j_As$ is a monomorphism. Analogously $a_2=b_2$, hence $a=b$ and $s + 1_B$ is proven to be a monomorphism.

Let $\arrow{S}{s}{A}$ be a subobject of $A$ and let $\arrow{T}{t}{A+B}$ be a subobject of $A+B$. Then the conditions $(A,j_A^*t)={j_A}^*(A+B,t)\leq (A,s)$ and $(A+B,t) \leq \prod_{j_A}(A,s)=(A+B,s+1_B)$ are equivalent: if the latter holds (i.e. $t \subseteq s+ 1_B$) then a mono witnessing the former (i.e. $j^*_At \subseteq s$) exists by the universal property of the pullback; viceversa, if the former holds for a mono $\arrow{{j_A}^*T}{m}{S}$ then $\arrow{T}{m+1_B}{S+B}$ witnesses the latter.
\end{example}

The previous examples highlight the fact that the existence of (weak) left and right adjoints over  reindexing, with respect to a given fibration, along injections is a natural requeriment: it is true in many concrete instances enjoying a satisfying ``power-set" algebra. For a further evidence of the generality of our hypotheses see Proposition \ref{ghesboro} below.

\begin{theorem}\label{summary preservations simple coproduct completion}
If $\fibration{\mE}{\mp}{\mB}$ is an extendable fibration, then $\coprodcomp{\mp}$ is an extendable fibration (with simple coproducts). Moreover, $\coprodcomp{\mp}$ satisfies the following:
\begin{enumerate}
     \item[$2.a)$] the total category $\coprodcomp{\mp}$ has both weak finite products and coproducts;
    \item[$2.b)$] if $\mp$ has fibred finite products, then $\coprodcomp{\mp}$ has fibred finite products and, if $\mp$ has fibred finite coproducts and left adjoints to the cartesian liftings along the injections, then $\coprodcomp{\mp}$ has fibred finite coproducts.
\end{enumerate}

\end{theorem}

\begin{proof}
First, let us verify $\coprodcomp{\mp}$ to be an extendable fibration, i.e. that it satisfies points 1), 2), 3) and 4) of Definition \ref{definition godel fibration}.

\medskip

$1)$ The base category of a fibration is not changed by the coproduct completion, hence the first point is trivial.

\medskip

$2)$ The preservation of fibred weak finite products follows by Proposition \ref{fibredproductpreservation}, while the proof that $\coprodcomp{\mp}$ has fibred weak finite coproducts looks like the proof of Proposition \ref{falso!} and the weak finite coproducts have the same presentation. However we need to observe that the properties of the adjunctions $\coprod_{j}\dashv j^*$ that we use there are enjoyed by the corresponding right-weak adjunctions $\coprod_{j}\radjunction j^*$, whose existence we have required in the statement of the current result, as well. In other words, we show that the properties of the right-weak adjunction $\coprod_{j}\radjunction j^*$ are enough to completely retrace the proof of Proposition \ref{falso!}.

If $I$ is an object of $\mB$ and $(I,X,\alpha)$ and $(I,Y,\beta)$ are objects of $\coprodcomp{\mp}_I$, then we claim that:\textit{ the object: $$\Big(I, X+Y, \clift{\big(\theta^{-1}\big)}\bigg(\coprod_{j_{I\x X}}\alpha+ \coprod_{j_{I\x Y}}\beta\bigg)=:\chi \Big)$$ where $\arrow{(I\x X)+(I\x Y)}{\theta}{I\x (X+Y)}$ is the canonical isomorphism, is a weak coproduct $(I,X,\alpha)+(I,Y,\beta)$ in $\coprodcomp{\mp}_I$ together with the injections:
$$(j_X\pr_X=\pi_{X+Y}\theta j_{I\times X},\;\ovln{j}_{\alpha}\eta_{\alpha})\text{ and }(j_Y\pr_Y=\pi_{X+Y}\theta j_{I\times Y},\;\ovln{j}_{\beta}\eta_{\beta})$$ where $\ovln{j}_\alpha\eta_\alpha$ and $\ovln{j}_\beta\eta_\beta$ are the compositions: $$\arrow{\alpha}{\eta_\alpha}{j^*_{I\times X}\coprod_{j_{I\times X}}\alpha\overset{\overline{j}_\alpha}}{\hookrightarrow}(\theta j_{I\times X})^*\chi=\langle \pi_I, j_X\pi_X \rangle^*\chi=(\theta j_{I\times X})^*\chi$$ and: $$\arrow{\beta}{\eta_\beta}{j^*_{I\times Y}\coprod_{j_{I\times Y}}\beta\overset{\overline{j}_\beta}}{\hookrightarrow}(\theta j_{I\times Y})^*\chi=\langle \pi_I, j_Y\pi_Y \rangle^*\chi=(\theta j_{I\times Y})^*\chi.$$ Here $\eta_\alpha$ and $\eta_\beta$ are components of the units of the right-weak adjunctions: $$\coprod_{j_{I\times X}}\radjunction j_{I\times X}^*\text{ and }\coprod_{j_{I\times Y}}\radjunction j_{I\times Y}^*$$ respectively (see Appendix \ref{weak adjunctions}). }

In order to prove this claim, let us assume that the couples: $$(\;f,\;\arrow{\alpha}{\phi}{\langle\pi_I, f\rangle^*\gamma=j_{I\times X}^*\langle \pi_I\theta,[f,g]\rangle^*\gamma}\;) \text{ and } (\;g,\;\arrow{\alpha}{\psi}{\langle\pi_I, g\rangle^*\gamma=j_{I\times Y}^*\langle \pi_I\theta,[f,g]\rangle^*\gamma}\;)$$ are arrows $(I,X,\alpha)\to(I,Z,\gamma) \text{ and }(I,Y,\beta)\to(I,Z,\gamma)$ and let us consider the $\sharp$-transposes $\phi^\sharp$ and $\psi^\sharp$ of $\phi$ and $\psi$ respectively along the right-weak adjunctions: $$\coprod_{I\times X}\radjunction j_{I\times X}^*\text{ and }\coprod_{I\times Y}\radjunction j_{I\times Y}^*$$ respectively. Then the couple: $$\big(\; [f,g]\theta^{-1},\; \chi \xrightarrow{(\theta^{-1})^*[\phi^\sharp,\psi^\sharp]}\langle \pi_I,[ f,g ]\theta^{-1}\rangle^*\gamma \;\big)$$ is an arrow $(I,X\times Y,\chi) \to (I,Z,\gamma)$ making the usual triangles commute. In fact it is the case that: $$\begin{aligned} ([f,g]\theta^{-1}, (\theta^{-1})^*[\phi^\sharp,\psi^\sharp])(j_X\pi_X,\ovln{j_\alpha}\eta_\alpha)&=([f,g]\theta^{-1}\theta j_{I\times X},((\theta j_{I\times X})^*(\theta^{-1})^*[\phi^\sharp,\psi^\sharp])\ovln{j_\alpha}\eta_\alpha) \\ &=(f,(j_{I\times X}^*\phi^\sharp)\eta_\alpha) \\ [\text{Proposition \ref{weak unit equality}}]&= (f,\phi)\end{aligned}$$ and analogously $([f,g]\theta^{-1}, (\theta^{-1})^*[\phi^\sharp,\psi^\sharp])(j_Y\pi_Y,\ovln{j_\beta}\eta_\beta)=(g,\psi)$.

The facts that the weak initial object of $\coprodcomp{\mp}_I$ exists and that weak finite coproducts of $\coprodcomp{\mp}_I$ are stable under $h$-reindexing, where $h$ is an arrow of $\mB$ of target $I$, follow as in Proposition $\ref{falso!}$.
\medskip

$3)$ Let us consider an object $(A,C,\gamma)$ of $\coprodcomp{\mp}_A$ and let $\arrow{A}{j_A}{A+B}$ be an injection. We define the object $\coprod^\coprodsym_{j_A}(A,C,\gamma)$ as the triple $(A+B,C,(\theta^{-1})^*\coprod_{j_{A\times C}}\gamma)$. We need to prove that the functor $\coprod^\coprodsym_{j_A}$ is right-weakly left adjoint to the $j_A$-reindexing w.r.t. $\coprodcomp{\mp}$. Let $(A+B,D,\delta)$ be an object of $\coprodcomp{\mp}_{A+B}$. Then its $j_A$-reindexing is the triple $(A,D,\langle j_A\pi_A,\pi_D\rangle^*\delta)$. Whenever $\big(\;f, \; (\theta^{-1})^*\coprod_{j_{A\times C}}\gamma \xrightarrow{\phi}\langle \pi_{A+B},f\rangle^*\delta\;\big)$ is an arrow: $$(A+B,C,(\theta^{-1})^*\coprod_{j_{A\times C}}\gamma)\to(A+B,D,\delta)$$ we define $\flat^{\coprodsym} (f,\phi)$ to be the arrow $(f\theta j_{A\times C}, (\theta^*\phi)^\flat)$, which is indeed an arrow: $$(A,C,\gamma)\to(A,D,\langle j_A\pi_A,\pi_D\rangle^*\delta).$$ Viceversa, whenever $\big(\;g,\; \gamma \xrightarrow{\psi}\langle \pi_A,g\rangle^*\langle j_A\pi_A,\pi_D \rangle^*\delta \;\big)$ is an arrow: $$(A,C,\gamma)\to(A,D,\langle j_A\pi_A,\pi_D\rangle^*\delta)$$ we define $\sharp^{\coprodsym}(g,\psi)$ to be the arrow $([g,d!]\theta^{-1},(\theta^{-1})^*\psi^\sharp)$, which is indeed an arrow: $$(A+B,C,(\theta^{-1})^*\coprod_{j_{A\times C}}\gamma)\to(A+B,D,\delta).$$ It is the case that: $$\begin{aligned}\flat^{\coprodsym}\sharp^{\coprodsym}(g,\psi)&=\flat^{\coprodsym}([g,d!]\theta^{-1},(\theta^{-1})^*\psi^\sharp) \\ &=([g,d!]\theta^{-1}\theta j_{A\times C},(\theta^*(\theta^{-1})^*\psi^\sharp)^\flat) \\ [\text{since }\flat\text{ is a retraction of }\sharp]&= (g,\psi)
\end{aligned}$$ hence $\coprod_{j_A}^\coprodsym$ is right-weakly left adjoint to the $j_A$-reindexing w.r.t. $\coprodcomp{\mp}$ and we are done. We are left to verify that BCC is satisfied.

Let us assume that the square: $$\xymatrix{C \ar[r]^{j_C} \ar[d]_f & C+D \ar[d]^{g} \\ A \ar[r]_{j_A} & A+B }$$ is a pullback and let $(A,E,\epsilon)$ be an object of $\coprodcomp{\mp}_A$. We are left to prove that: the objects $\coprod^\coprodsym_{j_C}f^*(A,E,\epsilon)=(C+D,E,(\theta^{-1})^*\coprod_{j_{C\times E}}(f\times 1_E)^*\epsilon)$ and $g^*\coprod^\coprodsym_{j_A}(A,E,\epsilon)=(C+D,E,(g\times 1_E)^*(\varphi^{-1})^*\coprod_{j_{A\times E}}\epsilon)$ are equal, where $\theta$ and $\varphi$ are the isomorphisms $C\times E + D\times E \to (C+ D)\times E$ and $A\times E + B\times E \to (A+B)\times E$ respectively. Let us consider the following commutative diagram: $$\xymatrix@-0.7pc{C\times E \ar[dd]_{f\times 1_E} \ar[rr]^>>>>>>>>{j_{C\times E}} && C\times E + D\times E \ar[dd]_{\varphi^{-1}(g\times 1_E)\theta} && (C+D)\times E \ar[dd]^{g\times 1_E} \ar[ll]_{\theta^{-1}} \ar[rr]^>>>>>>>>{\pr_{C+D}} && C+D \ar[dd]^{g} \\ \\  A\times E \ar[rr]_>>>>>>>>{j_{A\times E}} && A\times E + B\times E && (A+B)\times E \ar[ll]^{\varphi^{-1}} \ar[rr]_>>>>>>>>{\pr_{A+B}} && A+B}$$ which is a pullback, since its horizontal arrow $A\times E \to A + B$ equals the arrow $j_A\pr_A$, the arrow $f\times 1_E$ is the pullback of $f$ along $\pr_A$ and $f$ is the pullback of $g$ along $j_A$ (hence $f\times 1_E$ is indeed the pullback of $g$ along $j_A\pr_A$). As the right-hand square is pullback, we deduce that the left-hand square is a pullback as well. Therefore, by BBC for the right-weak adjunctions $\coprod_{j}\radjunction j^*$, it is the case that: $$\begin{aligned} (\theta^{-1})^*\coprod_{j_{C\times E}}(f\times 1_E)^*\epsilon&= (\theta^{-1})^*(\varphi^{-1}(g\times 1_E)\theta)^*\coprod_{j_{A\times E}}\epsilon \\ &=(g\times 1_E)^*(\varphi^{-1})^*\coprod_{j_{A\times E}}\epsilon \end{aligned}$$ and we are done.

\medskip
$4)$ As usual (see \cite{trottapasquale2020}), we define $\prod_{j_A}^\coprodsym(A,C,\gamma)$ to be the object: $$(A+B,C,(\theta^{-1})^*\prod_{j_{A\times C}}\gamma)$$ of $\coprodcomp{\mp}$ whenever $(A,C,\gamma)$ is an object of $\coprodcomp{\mp}_A$ and $\arrow{A}{j_A}{A+B}$ is an injection. Whenever $(A+B,D,\delta)$ is an object of $\coprodcomp{\mp}_{A+B}$ and the pair $(g,\; \arrow{\delta}{\phi}{\langle \pi_{A+B},g\rangle^*(\theta^{-1})^*\prod_{j_{A\times C}}\gamma})$ is an arrow $(A+B,D,\delta) \to (A+B,C,(\theta^{-1})^*\prod_{j_{A\times C}}\gamma)$, we define the arrow $\sharp^\coprodsym(g,\phi)$ to be the couple: $$(g\langle j_A\pi_A,\pi_D\rangle, \varepsilon_{\langle\pi_A,g\langle j_A\pi_A,\pi_D\rangle \rangle^*\gamma}\langle j_A\pi_A,\pi_D\rangle^*\phi)$$which is actually an arrow $(A,D,\langle j_A\pi_A,\pi_D \rangle^*\delta)\to(A,C,\gamma)$, whose second component is the composition of the arrows: \begin{itemize}
    \item $\langle j_A\pi_A,\pi_D\rangle^*\delta \xrightarrow{\langle j_A\pi_A,\pi_D\rangle^*\phi}\langle j_A\pi_A,\pi_D\rangle^*\langle\pi_{A+B},g \rangle^*(\theta^{-1})^*\prod_{j_{A\times C}}\gamma$
    \item $\langle j_A\pi_A,\pi_D\rangle^*\langle\pi_{A+B},g \rangle^*(\theta^{-1})^*\prod_{j_{A\times C}}\gamma=\langle \pi_A,g\langle j_A\pi_A,\pi_D\rangle\rangle^*j_{A\times C}^*\prod_{j_{A\times C}}\gamma$
    \item $\langle \pi_A,g\langle j_A\pi_A,\pi_D\rangle\rangle^*j_{A\times C}^*\prod_{j_{A\times C}}\gamma=j_{A\times D}^*\phi^*\langle \pi_{A+B},g \rangle^*(\theta^{-1})^*\prod_{j_{A\times C}}\gamma$
    \item $j_{A\times D}^*\phi^*\langle \pi_{A+B},g \rangle^*(\theta^{-1})^*\prod_{j_{A\times C}}\gamma=j_{A\times D}^*\prod_{j_{A\times D}}\langle \pi_A,g\langle j_A\pi_A,\pi_D\rangle \rangle^*\gamma$
    \item $j_{A\times C}^*\prod_{j_{A\times D}}\langle \pi_A,g\langle j_A\pi_A,\pi_D\rangle \rangle^*\gamma \xrightarrow{ \varepsilon_{\langle\pi_A,g\langle j_A\pi_A,\pi_D\rangle \rangle^*\gamma}}\langle\pi_A,g\langle j_A\pi_A,\pi_D\rangle \rangle^*\gamma$
\end{itemize} where the equality at the third bullet holds because the commutative square: $$\xymatrix{A\times D \ar[d]_{\langle \pi_A,g\langle j_A\pi_A,\pi_D\rangle\rangle} \ar[rr]^>>>>>>>>>{j_{A\times D}} && (A\times D)+(B\times D) \ar[d]^{\theta^{-1}\langle \pi_{A+B},g\rangle \phi} \\ A\times C \ar[rr]^>>>>>>>>>{j_{A\times C}} && (A\times C)+(B\times C)}$$ is a pullback.

Viceversa, whenever $(h,\; j_{A\times D}^*\phi^*\delta=\langle j_A\pi_A,\pi_D\rangle^*\delta\xrightarrow{\psi}\langle \pi_A,h \rangle^*\gamma)$ is an arrow: $$(A,D,\langle j_A\pi_A,\pi_D\rangle^*\delta)\to(A,C,\gamma)$$ we define $\flat^\coprodsym(h,\psi)$ to be the couple $([h,c!]\phi^{-1},(\phi^{-1})^*\psi^\flat)$ which is actually an arrow $(A+B,D,\delta) \to (A+B,C,(\theta^{-1})^*\prod_{j_{A\times C}}\gamma)$. Observe indeed that $\psi^\flat$ is an arrow $\phi^*\delta \to \prod_{j_{A\times S}}\langle \pi_{A},h\rangle^*\gamma$, hence $(\phi^{-1})^* \psi^\flat$ is an arrow: $$\delta \to \phi^{-1}\prod_{j_{A\times D}}\langle \pi_{A},h\rangle^*\gamma=\langle \pi_{A+B},[h,c!]\phi^{-1}\rangle^*(\theta^{-1})^*\prod_{j_{A\times C}}\gamma$$ where the last equality holds because the commutative square: $$\xymatrix{A\times D \ar[d]_{\langle \pi_A,h\rangle} \ar[rr]^>>>>>>>>>{j_{A\times D}} && (A\times D)+(B\times D) \ar[d]^{\theta^{-1}\langle \pi_{A+B},[h,c!]\phi^{-1}\rangle \phi} \\ A\times C \ar[rr]^>>>>>>>>>{j_{A\times C}} && (A\times C)+(B\times C)}$$ is a pullback. It is the case that: $$\begin{aligned} \sharp^\coprodsym\flat^\coprodsym (h,\psi) &=  \sharp^\coprodsym([h,c!]\phi^{-1},\;(\phi^{-1})^*\psi^\flat) \\ &=([h,c!]\phi^{-1}\langle j_A\pi_A,\pi_D\rangle,\;\varepsilon_{\langle \pi_A,[h,c!]\phi^{-1}\langle j_A\pi_A,\pi_D\rangle \rangle^*\gamma}\langle j_A\pi_A,\pi_D\rangle^*(\phi^{-1})^*\psi^\flat) \\
&= (h,\; \varepsilon_{\langle\pi_A,h \rangle^*\gamma}(j_{A\times D}^*\psi^\flat) ) \\ [\text{Proposition \ref{weak counit equality}}] &=(h,\psi) \end{aligned}$$ hence $\prod_{j_A}^\coprodsym$ is left-weakly right adjoint to the $j_A$-reindexing w.r.t. $\coprodcomp{\mp}$ and we are done. The fact that BCC is satisfied follows precisely as it follows in 3).

\medskip

Finally, let us verify that $2.a)$ and $2.b)$ hold.

\medskip
$2.a)$ The existence of (weak) products follow by (a weak version of) \cite[Prop. 9.2.1]{Jacobs1999}. The tricky point is to use the existence of (weak) left adjoints and fibred (weak) coproducts to define (weak) coproducts in the total category $\coprodcomp{\mp}$.

Let $(I,A,\alpha)$ and $(Y,B,\beta)$ be two objects of the total category $\coprodcomp{\mp}$. We define
\[ (I,A,\alpha)+(Y,B,\beta):= \coprod_{j_I}^\coprodsym(I,A,\alpha)+\coprod_{j_Y}^\coprodsym(Y,B,\beta)\]
i.e.
\[(I,A,\alpha)+(Y,B,\beta)=(I+Y,A+B, \omega^{*} (\coprod_{j_{I\x A}}(\alpha)+\coprod_{j_{Y\x B}}(\beta)))\]
where $\omega$ is the canonical isomorphism $(I+Y)\x (A+B)\cong (I\x A)+ (I\x B)+ (Y\x A)+ (Y\x B)$, together with injections:
\[(j_I,j_A\pr_A, \ovln{j}_{\alpha} \eta_{\alpha})\mbox{ and } (j_Y,j_B\pr_B,\ovln{j}_{\beta}  \eta_{\beta}).\]
It is direct to verify that this is actually a weak coproduct.

\medskip
$2.b)$ This is just Proposition \ref{fibredproductpreservation} and Proposition \ref{falso!}.\qedhere

\end{proof}

We state the dual version of Theorem \ref{summary preservations simple coproduct completion}, which follows by Lemma \ref{opfibration}, Proposition \ref{proposition coproc is dual to prod} and Theorem \ref{summary preservations simple coproduct completion}.

\begin{theorem}\label{summary preservations simple product completion}
If $\fibration{\mE}{\mp}{\mB}$ is an extendable fibration, then $\prodcomp{\mp}$ is an extendable fibration (with simple products). Moreover, $\prodcomp{\mp}$ satisfies the following:
\begin{enumerate}
\item[$2.a)$] the total category $\prodcomp{\mp}$ has both weak finite products and coproducts;
    \item[$2.b)$] if $\mp$ has fibred finite coproducts, then $\prodcomp{\mp}$ has fibred finite coproducts and, if $\mp$ has fibred finite products and right adjoints to the cartesian liftings along the injections, then $\prodcomp{\mp}$ has fibred finite products.
    \end{enumerate}
\end{theorem}

Observe that Theorem \ref{summary preservations simple coproduct completion} and Theorem \ref{summary preservations simple product completion} show that not only extendable fibrations are preserved by the simple product and coproduct completions, but that these constructions extend the structures in the fibres to the whole total category.

\begin{theorem}\label{existential preserves universal}
Let $\mB$ be a cartesian closed category and let $\fibration{\mE}{\mp}{\mB}$ be a fibration with simple products. Then $\fibration{\coprodcomp{\mp}}{\coprodcomp{\mp}}{\mB}$ is a fibration with simple products and simple coproducts, i.e. the simple coproduct completion preserves the existence of simple product.

\proof \text{(Hofstra \cite{hofstra2011}) }

\medskip

\textit{Part I. Existence of right adjoints to weakenings.} Let $A_1,A_2$ be objects of $\mB$, and let $\arrow{A_1\x A_2}{\pr_{A_1}}{A_1}$ be the first projection. Let: $$\arrow{\coprodcomp{\mp}_{A_1\x A_2}}{\prod^\coprodsym_{\pi_{A_1}}}{\coprodcomp{\mp}_{A_1}}$$ be defined by: $$ (A_1\x A_2,B,\alpha)\mapsto (A_1,B^{A_2},\prod_{\angbr{\pr_1}{\pr_3}} {\angbr{\pr_1,\pr_2}{\ev \angbr{\pr_2}{\pr_3}}}^*\alpha)$$ where $\pr_i$ are the projections from $A_1\x A_2 \x B^{A_2}$ and $\arrow{A_2 \x B^{A_2}}{\ev}{B}$ is the evaluation map. The intuition is that the right adjoints act by mapping a formula $\exists b:B \alpha (a_1,a_2,b)\mapsto \exists f:B^{A_2}\forall a_2:A_2 \alpha(a_1,a_2, f(a_2))$. Let us verify that $\prod^\coprodsym_{\pr_{A_1}}$ is right adjoint to the $\pi_{A_1}$-weakening w.r.t. $\coprodcomp{\mp}$.

Let $(A_1,C,\gamma)$ and $(A_1\times A_2,B,\alpha)$ be objects of $\coprodcomp{\mp}_{A_1}$ and $\coprodcomp{\mp}_{A_1\times A_2}$ respectively. We are left to verify that between the homsets:
$$\coprodcomp{\mp}_{A_1\times A_2}\Big(\;(A_1 \times A_2,C,({\pr_{A_1}\times 1_C})^*(\gamma)),\;(A_1\times A_2,B,\alpha)\;\Big)$$
and:
$$\coprodcomp{\mp}_{A_1}\Big(\;(A_1,C,\gamma),\;(A_1,B^{A_2},\prod_{\angbr{\pr_1}{\pr_3}} {\angbr{\pr_1,\pr_2}{\ev \angbr{\pr_2}{\pr_3}}}^*\alpha\;\Big)$$
there is a natural bijection. An arrow of the former consists of: \textit{an arrow} $\arrow{A_1\times A_2\times C}{g}{B}$ \textit{together with an arrow} $({\pr_{A_1}\times 1_C})^*(\gamma) \xrightarrow{\varphi} {\langle \pr_{A_1\times A_2},g\rangle}^*\alpha$, that is:
\begin{align}
&\textit{an arrow }\arrow{A_1\times A_2\times C}{g}{B} \notag\\ &\textit{together with an arrow }\gamma \xrightarrow{\phi} \prod_{\pr_{A_1}\times 1_C}{\langle \pr_{A_1\times A_2},g\rangle}^*\alpha
\end{align}
while an arrow of the latter consists of:
\begin{align}
&\textit{an arrow }\arrow{A_1\times C}{h}{B^{A_2}} \notag\\ &\textit{together with an arrow } \gamma \xrightarrow{\psi} {\langle \pr_{A_1},h\rangle^*}\prod_{\angbr{\pr_1}{\pr_3}} {\angbr{\pr_1,\pr_2}{\ev \angbr{\pr_2}{\pr_3}}}^*\alpha
\end{align}
so that we are left to prove that there is a natural bijection between couples as in (2) and couples as in (3). Let $\pr'_i$ be the three projections from $A_2\times A_1 \times B^{A_2}$, $\overline{\pr}_i$ the projections from $A_1 \times A_2 \times C$ and $\overline{\pr}'_i$ the projections from $A_2\times A_1 \times C$. Moreover let us assume that there is a couple $(g,\phi)$ as in (2). We define the arrow $\arrow{A_1 \times C}{h}{B^{A_2}}$ to be the exponential transpose of $g\langle \overline{\pr}'_2,\overline{\pr}'_1,\overline{\pr}'_3\rangle$, hence it holds that $\ev (1_{A_2}\times h)=g\langle \overline{\pr}'_2,\overline{\pr}'_1,\overline{\pr}'_3\rangle$. We are going to prove that this choice of $h$ is such that:
\begin{align}
\prod_{\pr_{A_1}\times 1_C} {\langle \pr_{A_1\times A_2},g\rangle}^*={\langle \pr_{A_1},h\rangle}^*\prod_{\angbr{\pr_1}{\pr_3}} {\angbr{\pr_1,\pr_2}{\ev \angbr{\pr_2}{\pr_3}}}^*
\end{align}
allowing to conclude that $(h,\psi:=\phi)$ is a couple as in (3). Moreover observe that, from a given couple $(h,\psi)$ satisfying (3), one can always recover the corresponding arrow $g$ by anti-transposing $h$ and precomposing by $\langle \overline{\pr}_2,\overline{\pr}_1,\overline{\pr}_3\rangle$ (as $\langle \overline{\pr}'_2,\overline{\pr}'_1,\overline{\pr}'_3\rangle$ is an isomorphism whose inverse is indeed $\langle \overline{\pr}_2,\overline{\pr}_1,\overline{\pr}_3\rangle$). Therefore (4) would also imply that $(g,\phi:=\psi)$ is a couple as in (2), concluding our adjointness proof. Hence we are left to prove that (4) holds.

Let $x$ be the arrow $\arrow{A_1\times A_2 \times C}{(1_{A_2}\times h)\langle \overline{\pr}_2,\overline{\pr}_1,\overline{\pr}_3 \rangle}{A_2 \times B^{A_2}}$, and observe that the equality: $$\langle \pr_1,\pr_2,\ev \langle \pr_2,\pr_3\rangle\rangle\langle \overline{\pr}_1,x \rangle =\langle \overline{\pr}_1,\overline{\pr}_2,g\rangle =\langle \pr_{A_1\times A_2},g \rangle$$ holds. Hence it is the case that ${\langle \pr_{A_1\times A_2},g \rangle}^*={\langle \overline{\pr}_1,x \rangle}^*{\langle \pr_1,\pr_2,\ev \langle \pr_2,\pr_3\rangle\rangle}^*$ and therefore (4) follows if we prove that $\prod_{\pr_{A_1}\times 1_C}{\langle \overline{\pr}_1,x\rangle}^*={\langle \pr_{A_1},h\rangle}^*\prod_{\angbr{\pr_1}{\pr_3}}$ holds. By BCC for $\prod$ it is enough to prove that the right-hand square of the commutative diagram: $$\xymatrix{A_2\times A_1\times C \ar[d]^{\langle \overline{\pr}'_1, \overline{\pr}'_2,h\langle \overline{\pr}'_1,\overline{\pr}'_3\rangle \rangle} \ar[rr]^{\langle \overline{\pr}'_2,\overline{\pr}'_1,\overline{\pr}'_3\rangle} && A_1\times A_2\times C \ar[rrr]^{\pr_{A_1}\times 1_C=\langle \overline{\pr}_1,\overline{\pr}_3 \rangle} \ar[d]^{\langle \overline{\pr}_1,x \rangle} &&& A_1 \times C \ar[d]^{\langle \pr_{A_1},h\rangle} \\ A_2\times A_1 \times B^{A_2} \ar[rr]_{\langle \pr'_2, \pr'_1, \pr'_3\rangle} && A_1\times A_2\times B^{A_2} \ar[rrr]_{\langle \pr_1,\pr_3\rangle}   &&& A_1 \times B^{A_2} }$$ is a pullback. This is the case: the outer square is a pullback, as its horizontal arrows are the projections $A_2 \times A_1 \times C \to A_1 \times C$ and $A_2 \times A_1 \times B^{A_2} \to A_1 \times B^{A_2}$ and as $\langle \overline{\pr}'_1, \overline{\pr}'_2,h\langle \overline{\pr}'_1,\overline{\pr}'_3\rangle \rangle=1_{A_2}\times \langle \pr_{A_1},h\rangle$, and moreover the horizontal arrows of the left-hand square are isos, therefore the right-hand square is indeed a pullback as well. 

\medskip

\textit{Part II. BCC.} Let us consider a pullback of a projection along a given arrow $f$, which is of the form: 
$$\quadratocomm{D \times C}{D}{A\times C}{A.}{\pr_D}{f \times 1_C}{f}{\pr_A}$$ 
and let us verify that the corresponding equality $f^*\prod^\coprodsym_{\pr_A}=\prod^\coprodsym_{\pr_D}{(f \times 1_C)}^*$ holds.  Whenever $(A\times C,B,\beta)$ is an object of $\coprodcomp{\mp}_{A\times C}$ we get (by applying the left and right member of the wannabe equality respectively) the elements: $$(D,B^C,{(f\times 1_{B^C})}^*\prod_{\langle\pr_1,\pr_3 \rangle}{\langle\pr_1,\pr_2,\ev\langle \pr_2,\pr_3\rangle \rangle}^*\beta)$$ and $$(D,B^C,\prod_{\langle\overline{\pr}_1,\overline{\pr}_3 \rangle}{\langle\overline{\pr}_1,\overline{\pr}_2,\ev\langle \overline{\pr}_2,\overline{\pr}_3\rangle \rangle}^*{(f\times 1_{C\times B})}^*\beta)$$ of $\coprodcomp{\mp}_D$, being $\pr_i$ the projections from $A\times C\times B^C$ and $\overline{\pr}_i$ the projections from $D \times C\times B^C$. We are left to prove them to be equal. By BCC for $\prod$ it is the case that ${(f\times 1_{B^C})}^*\prod_{\langle\pr_1,\pr_3 \rangle}=\prod_{\langle \overline{\pr}_1,\overline{\pr}_3\rangle}{(f\times 1_C \times 1_{B^C})}^*$, hence we are left to observe that: $${(f\times 1_C\times 1_{B^C})}^*{\langle\pr_1,\pr_2,\ev\langle \pr_2,\pr_3\rangle \rangle}^*={\langle\overline{\pr}_1,\overline{\pr}_2,\ev\langle \overline{\pr}_2,\overline{\pr}_3\rangle \rangle}^*{(f\times 1_{C\times B})}^*$$ which holds because the class of arrows: $$\arrow{X \times C\times B^C}{\langle\pr_1,\pr_2,\ev\langle \pr_2,\pr_3\rangle \rangle}{X \times C\times B}$$ for $X$ in $\mB$ is a natural transformation $(-)\times C \times B^C\to (-)\times (C\times B)$.
\endproof
\end{theorem}

We show that under our hypotheses the Dialectica monad $\dial{-}$ preserves the existence of fibred weak finite products and coproducts of a given fibration. The re-indexing along injections of the resulting fibration $\dial{\mp}$ happen to have weakly left and right adjoints as well.
Moreover, the total category $\dial{\mp}$ has finite (weak) products and coproducts.

\begin{theorem}\label{theorem dial has weak products and coproducts}
Let $\mB$ be a distributive category and let $\fibration{\mE}{\mp}{\mB}$ be an extendable fibration. 
Then: $$\fibration{\dial{\mp}}{\dial{\mp}}{\mB}$$
is an extendable G\"odel fibration and the total category $\dial{\mp}$ has both weak finite products and coproducts.

Moreover, if $\mp$ has fibred finite products and right adjoints to the cartesian liftings along the injections, then the fibration $\dial{\mp}$ has fibred finite products and, hence, its total category $\dial{\mp}$ has finite products.

\proof
It follows by Theorem \ref{summary preservations simple coproduct completion} and Theorem \ref{summary preservations simple product completion}. In detail, the first part of the statement follows by $2.a)$ of Theorem \ref{summary preservations simple coproduct completion} and $2.a)$ of Theorem \ref{summary preservations simple product completion}, and the second part of the statement follows by the second part of $2.b)$ of Theorem \ref{summary preservations simple product completion} and the first part of $2.b)$ of Theorem \ref{summary preservations simple coproduct completion}.
\endproof
\end{theorem}

Combining Theorem \ref{summary preservations simple coproduct completion} together with Theorem \ref{theorem godel if and only if Dialectica} and Theorem \ref{existential preserves universal}, the following result follows:

\begin{cor}
Let $\mB$ be a cartesian closed category and let $\fibration{\mE}{\mp}{\mB}$ be an extendable fibration. 
Then: $$\fibration{\dial{\mp}}{\dial{\mp}}{\mB}$$
is an extendable G\"odel fibration (hence in particular an existential and universal fibration) and the total category $\dial{\mp}$ has both weak finite products and coproducts.

Moreover, if $\mp$ has fibred finite products and right adjoints to the cartesian liftings along the injections, then the fibration $\dial{\mp}$ has fibred finite products and, hence, its total category $\dial{\mp}$ has finite products.
\end{cor}

We conclude this section by comparing our results with the literature.  Hyland \cite{HYLAND200243} and  Biering \cite{BIERING2008290} described how to go beyond simple tensor logic, providing sufficient conditions to ensure the existence of finite products in the Dialectica category associated to a cloven fibration. Hyland's formulation for preordered fibrations identified the following key requirements: 
\begin{itemize}
    \item the base category of a fibration $\fibration{\mE}{\mp}{\mB}$ has to be cartesian closed, with finite coproducts;
    \item $\mp$ has to be fibred cartesian closed;
    \item the injections $\arrow{A}{j_A}{A+B}$ and $\arrow{B}{j_B}{A+B}$ have to induce an equivalence $ \mE_A \x \mE_B\cong \mE_{A+B}$, natural in $A$ and $B$, and also $\mE_0\cong \mathbf{1}$.
\end{itemize}
Under these hypotheses, Hyland proved that for $\mathbf{Dial}(\mp)$  the fibre $\dial{\mp}_1$, has finite products.

 Biering has also shown that the Dialectica category $\mathbf{Dial}(\mp)$ which coincides with the fibred category $\dial{\mp}_1$ of the fibration $\dial{\mp}$ used in \cite{hofstra2011} and here has products.

\begin{proposition}[Prop. 6 in \cite{BIERING2008290}]\label{proposition Biering}
Let $\fibration{\mE}{\mp}{\mB}$ be a cloven fibration. 
\begin{itemize}
    \item Suppose $\mB$ has finite, distributive products and coproducts, and that the injections $\arrow{A}{j_A}{A+B}$ and $\arrow{B}{j_B}{A+B}$ induce an equivalence $\mu := \mE_A \x \mE_B\cong \mE_{A+B}$, natural in $A$ and $B$, then $\mathbf{Dial}(\mp)$ has binary products.
    \item Suppose that $\mE_0\cong \mathbf{1}$, then $\mathbf{Dial}(\mp)$ has a terminal object.
\end{itemize}
\end{proposition}

Our goal is clearly similar, we want to know when $\dial{\mp}$ has (weak) products and (weak) coproducts, but our perspective is different. Since we take advantage of Hofstra's decomposition, which exploits the \emph{duality} of product and coproduct completions, the hypotheses we assume are always easily \emph{dualized}. 
We also emphasize the naturality of asking for the \emph{existence of adjoints} along the re-indexing by some maps, instead of the equivalence between the fibres they introduce. In the literature one can find several properties of fibrations which depend on the existence of particular adjoints. 

In the following proposition we formally compare our assumptions with those introduced previously.
\begin{proposition} \label{ghesboro}
Let $\fibration{\mE}{\mp}{\mB}$ be a fibration satisfying the conditions of \ref{proposition Biering}. If $\mp$ has fibred finite products, then $\mp$ has right adjoints to the re-indexing functors along injections. Similarly, if $\mp$ has fibred coproducts, then $\mp$ has left adjoints to the re-indexing functors along injections.
\end{proposition}
\begin{proof}
If a fibration $\fibration{\mE}{\mp}{\mB}$ has finite finite products, and it satisfies the hypothesis of Proposition \ref{proposition Biering}, for every $A$ and $B$, we can define a functor $\arrow{\mE_A}{}{\mE_{A+B}}$ as 
$$\xymatrix{\mE_A\ar[rr]^{\angbr{\id_{\mE_A}}{\top}}& &\mE_A\x \mE_B\ar[r]^{\mu} &\mE_{A+B}
} $$
where $\arrow{\mE_A}{\top}{\mE_B}$ is the functor which assigns to every object of the fibre $\mE_A$ the terminal object of $\mE_B$. We denote the composition $\mu \angbr{\id_{\mE_A}}{\top}$ by $\prod_{j_A}$.
Now we have that 
$$\mE_{A+B}(\alpha, \prod_{j_A}(\beta))\cong \mE_{A}\x\mE_B(\mu^{-1}(\alpha), \angbr{\id_{\mE_{A}}}{\top}(\beta))=\mE_A\x \mE_B( \angbr{j_A^*(\alpha)}{j_B^*(\beta)},\angbr{\beta}{\top})$$ 
and the last one is isomorphic to $\mE_A(j_A^*(\alpha),\beta)$. 
\end{proof}

\section{Conclusion}
Our results clarify the original Dialectica construction from both a categorical and logical perspective, and  they contribute to a deeper understanding of the construction. 

Our first main result Theorem \ref{theorem godel if and only if Dialectica}, provides an internal characterization of fibrations which are instances of the Dialectica construction, highlighting the key features a fibration should satisfy, namely it must be a G\"odel fibration, 
to be an instance of the Dialectica construction. 

Our presentation in terms of  G\"odel fibrations underlines a double nature of Dialectica fibrations: they satisfy principles which are typical of classical logic, such as the existence of a prenex normal form presentation for formulae, but they also satisfy principles normally associated to intuitionistic logic.  For example, they satisfy the existence of terms witnessing a proof: for every proof of  $\alpha \vdash \exists x \beta (x)$ where $\alpha$ is quantifier-free, we have a proof of $\alpha \vdash \beta (t)$ for some term $t$.  

We have proved that the Dialectica construction proposed by Hofstra \cite{hofstra2011} preserves the properties of extendable fibrations, showing that the Dialectica construction has a symmetric behaviour with respect to weak structures, but it is asymmetric with the strict ones.

Moreover, we have generalized the results of de Paiva \cite{dePaiva1989Dialectica}, Hyland \cite{HYLAND200243} and Biering \cite{BIERING2008290} to this general construction, and we have shown, in addition, that assuming coproducts in the starting extendable fibration, the weak coproducts in the corresponding Dialectica category are also associative.
This implies that the weak coproducts are almost as good as real coproducts, as far as logical properties are considered.


Dialectica-like constructions are pervasive in several areas of mathematics and computer science, and we briefly describe some future work, based on our previous analysis. 
We wonder if the decomposition introduced by Hofstra can be extended or modified to provide similar results for \emph{cousins} of the Dialectica construction. In particular, we  believe that this decomposition, combined with the results presented in \cite{trotta2020}, could be generalized to the context of dependent type theory.

There are two fibrations which seem to share common features with the Dialectica construction. In particular, we would like to investigate and compare the fibrations arising from work by Abramsky and Väänänen \cite{IfBi} on the Hodges semantics for independence-friendly logic and the Dialectica tripos, which is a model of separation logic \cite{Biering_Dialecticainterpretations}.

Finally, the strong constructive features of Dialectica fibrations we have shown suggest that these kinds of fibrations could lead to interesting applications in constructive foundations for mathematics and proof theory.
Recall that Maietti \cite{DBLPMaietti17} showed that the unique choice rule, and hence the choice rule, are not valid in either Coquand’s Calculus of Constructions nor in its predicative version implemented in the intensional level of the Minimalist Foundation. This means that in these theories the extraction of computational witnesses from existential statements must be performed in a more expressive proofs-as-programs theory.
However, given the preservation of the structures together with the validity of a form of choice in the internal language of the Dialectica fibration we provided, we may be able to use the Dialectica monad to extend constructive theories to more expressive systems, in which the extraction of computational witnesses from existential statements is allowed.
With this result in mind, understanding when the Dialectica construction preserves the logical implication is fundamental, since this would be required for extending the theory preserving all the logical constructors. 
\bibliographystyle{plain}
\bibliography{references}

\appendix
\section{Left-weak and right-weak adjunctions}\label{weak adjunctions}

\noindent
We recall that the notions of right-weak and left-weak adjunction appear in a similar context in \cite{moss2018}. 

\begin{definition}\label{right-weak adjunction}
Let $\mathcal{C}$ and $\mathcal{D}$ be functors and let $F \colon \mathcal{C}\longleftarrow\mathcal{D}$ and $\fibration{\mathcal{C}}{G}{\mathcal{D}}$ be functors. We say that the pair $(F,G)$ is a \bemph{right-weak adjunction} of the categories $\mathcal{C}$ and $\mathcal{D}$, and we indicate this as $F \radjunction G$, if there is a natural transformation: $$\mathcal{C}(F-,-)\xrightarrow{(-)^\flat}\mathcal{D}(-,G-)$$ together with a choice of a section $(-)^\sharp$ of every $(D,C)$-component of $(-)^\flat$, being $C$ an object of $\mathcal{C}$ and $D$ an object of $\mathcal{D}$. We also say that $F$ is \bemph{right-weakly left adjoint} to $G$ and that $G$ is \bemph{right-weakly right adjoint} to $F$.
\end{definition}

\begin{definition}\label{left-weak adjunction}
Let $\mathcal{C}$ and $\mathcal{D}$ be functors and let $F \colon \mathcal{C}\longleftarrow\mathcal{D}$ and $\fibration{\mathcal{C}}{G}{\mathcal{D}}$ be functors. We say that the pair $(F,G)$ is a \bemph{left-weak adjunction} of the categories $\mathcal{C}$ and $\mathcal{D}$, and we indicate this as $F \ladjunction G$, if there is a natural transformation: $$\mathcal{C}(F-,-)\xleftarrow{(-)^\sharp}\mathcal{D}(-,G-)$$ together with a choice of a section $(-)^\flat$ of every $(D,C)$-component of $(-)^\sharp$, being $C$ an object of $\mathcal{C}$ and $D$ an object of $\mathcal{D}$. We also say that $F$ is \bemph{left-weakly left adjoint} to $G$ and that $G$ is \bemph{left-weakly right adjoint} to $F$.
\end{definition}

\noindent
\textit{Part I. Properties of right-weak adjunctions.} Let $(\fibration{\mathcal{D}}{F}{\mathcal{C}})\radjunction (\fibration{\mathcal{C}}{G}{\mathcal{D}})$ be a right-weak adjunction of the categories $\mathcal{C}$ and $\mathcal{D}$ (see Definition \ref{right-weak adjunction}). Let $D$ be an object of $\mathcal{D}$ and let $\arrow{D}{\eta_D}{GFD}$ be the arrow $(\arrow{FD}{1_{FD}}{FD})^\flat$. Then the following results hold:

\begin{proposition}[\textit{Unit of $F\radjunction G$}] \label{q}
The class of the arrows $\eta_D$, where $D$ is an object of $\mathcal{D}$, is a natural transformation $\arrow{1_{\mathcal{D}}}{\eta}{GF}$, which is called \emph{unit} of the right-weak adjunction.

\proof
Whenever $\arrow{D}{f}{D'}$ is an arrow of $\mathcal{D}$, the equality $(GFf)\eta_D=\eta_{D'}f$ follows by the commutativity of the diagram: $$\xymatrix@-0.6pc{\mathcal{C}(FD',FD') \ar[d]_{(-)^\flat} \ar[rr]^{(-)(Ff)} && \mathcal{C}(FD,FD') \ar[d]_{(-)^\flat} && \mathcal{C}(FD,FD) \ar[ll]_{(Ff)(-)} \ar[d]_{(-)^\flat} \\ \mathcal{D}(D',GFD') \ar[rr]^{(-)(f)} && \mathcal{D}(D,GFD') && \mathcal{D}(D,GFD)\ar[ll]_{(GFf)(-)}}$$ applied to the arrows $1_{FD'}$ a $1_{FD}$.
\endproof
\end{proposition}

\begin{proposition}\label{weak unit equality}
Let us consider an arrow $\arrow{D}{g}{GC}$ of $\mathcal{D}$. Then the equality: $$G(g^\sharp)\eta_D=g$$ holds.

\proof
The equality follows by the commutativity of the diagram: $$\xymatrix{\mathcal{C}(FD,C) \ar[d]_{(-)^\flat} && \mathcal{C}(FD,FD) \ar[ll]_{(g^\sharp)(-)} \ar[d]_{(-)^\flat} \\ \mathcal{D}(D,GC) && \mathcal{D}(D,GFD)\ar[ll]_{(G(g^\sharp))(-)}}$$ applied to the arrow $1_{FD}$ and since $(g^\sharp)^\flat=g$. \endproof
\end{proposition}

\begin{proposition} \label{r-w r adjoint preserves w l}
The functor $G$ preserves weak limits.

\begin{proof}
Let $T$ be a functor $\mathcal{J}\to \mathcal{C}$ and let $\arrow{C}{\gamma}{T}$ be a $\mathcal{J}$-indexed class of arrows in $\mathcal{C}$ exhibiting $C$ as a weak limit of $T$ in $\mathcal{C}$. Whenever $\arrow{D}{\delta}{GT}$ is a $\mathcal{J}$-indexed class of arrows in $\mathcal{D}$, it is the case that: $$(\arrow{FD}{\delta^{\sharp}}{T}):=\{ \arrow{FD}{\delta_J^\sharp}{TJ} \}_{J \colon \mathcal{J}}$$ is a $\mathcal{J}$-indexed class of arrows in $\mathcal{C}$. Hence there is an arrow $\arrow{FD}{\phi}{C}$ such that the $\mathcal{J}$-indexed class of triangles: $$\xymatrix{FD \ar[r]^{\delta^\sharp} \ar[d]_{\phi} & T \\ C \ar[ur]_{\gamma}}$$ of $\mathcal{C}$ commutes component-wise. Therefore, by naturality of $(-)^\flat$ and being the components of $(-)^\sharp$ sections of the ones of $(-)^\flat$, it is the case that the $\mathcal{J}$-indexed class of triangles: $$\xymatrix{D \ar[r]^{\delta} \ar[d]_{\phi^\flat} & GT \\ GC \ar[ur]_{G\gamma}}$$ of $\mathcal{D}$ commutes component-wise. We conclude that the $\mathcal{J}$-indexed family $G\gamma$ of arrows of $\mathcal{D}$ exhibits $GC$ as a weak limit of $GT$ in $\mathcal{D}$.
\end{proof}
\end{proposition}

\noindent
\textit{Part II. Properties of left-weak adjunctions.} Analogously, let $(\fibration{\mathcal{D}}{F}{\mathcal{C}})\ladjunction (\fibration{\mathcal{C}}{G}{\mathcal{D}})$ be a left-weak adjunction of the categories $\mathcal{C}$ and $\mathcal{D}$ (see Definition \ref{left-weak adjunction}). Let $C$ be an object of $\mathcal{C}$ and let $\arrow{FGC}{\varepsilon_C}{C}$ be the arrow $(\arrow{GC}{1_{GC}}{GC})^\sharp$. Then the following analogous results hold:

\begin{proposition}[\textit{Counit of $F\ladjunction G$}] 
The class of the arrows $\varepsilon_C$, where $C$ is an object of $\mathcal{C}$, is a natural transformation $\arrow{FG}{\varepsilon}{1_{\mathcal{C}}}$, which is called \emph{counit} of the left-weak adjunction.

\proof
Whenever $\arrow{C}{g}{C'}$ is an arrow of $\mathcal{C}$, the equality $\varepsilon_{C'}(FGg)=g\varepsilon_C$ follows by the commutativity of the diagram: $$\xymatrix@-0.6pc{\mathcal{C}(FGC,C) \ar[rr]^{(g)(-)} && \mathcal{C}(FGC,C') && \mathcal{C}(FGC',C') \ar[ll]_{(-)(FGg)} \\ \mathcal{D}(GC,GC) \ar[u]^{(-)^\sharp} \ar[rr]^{(Gg)(-)} && \mathcal{D}(GC,GC') \ar[u]^{(-)^\sharp} && \mathcal{D}(GC',GC') \ar[u]^{(-)^\sharp}\ar[ll]_{(-)(Gg)}}$$ applied to the arrows $1_{GC}$ a $1_{GC'}$.
\endproof
\end{proposition}

\begin{proposition}
\label{weak counit equality}
Let us consider an arrow $\arrow{FD}{f}{C}$ of $\mathcal{D}$. Then the equality: $$\varepsilon_CF(f^\flat)=f$$ holds.

\proof
The equality follows by the commutativity of the diagram: $$\xymatrix{\mathcal{C}(FD,C) && \mathcal{C}(FGC,C) \ar[ll]_{(-)(F(f^\flat))} \\ \mathcal{D}(D,GC) \ar[u]^{(-)^\sharp} && \mathcal{D}(GC,GC) \ar[u]^{(-)^\sharp}\ar[ll]_{(-)(f^\flat)}}$$ applied to the arrow $1_{GC}$ and since $(f^\flat)^\sharp=f$.
\endproof
\end{proposition}

\begin{proposition} \label{l-w l adjoint preserves w c}
The functor $F$ preserves weak colimits.

\begin{proof}
Let $T$ be a functor $\mathcal{J}\to \mathcal{D}$ and let $\arrow{T}{\gamma}{D}$ be a $\mathcal{J}$-indexed class of arrows in $\mathcal{D}$ exhibiting $D$ as a weak colimit of $T$ in $\mathcal{D}$. Whenever $\arrow{FT}{\delta}{C}$ is a $\mathcal{J}$-indexed class of arrows in $\mathcal{C}$, it is the case that: $$(\arrow{T}{\delta^{\flat}}{GC}):=\{ \arrow{TJ}{\delta_J^\flat}{GC} \}_{J \colon \mathcal{J}}$$ is a $\mathcal{J}$-indexed class of arrows in $\mathcal{D}$. Hence there is an arrow $\arrow{D}{\phi}{GC}$ such that the $\mathcal{J}$-indexed class of triangles: $$\xymatrix{T \ar[r]^{\gamma} \ar[dr]_{\delta^\flat} & D \ar[d]^{\phi} \\ & GC}$$ of $\mathcal{D}$ commutes component-wise. Therefore, by naturality of $(-)^\sharp$ and being the components of $(-)^\flat$ sections of the ones of $(-)^\sharp$, it is the case that the $\mathcal{J}$-indexed class of triangles: $$\xymatrix{FT \ar[r]^{F\gamma} \ar[dr]_{\delta} & FD \ar[d]^{\phi^\sharp} \\ & C}$$ of $\mathcal{C}$ commutes component-wise. We conclude that the $\mathcal{J}$-indexed family $F\gamma$ of arrows of $\mathcal{C}$ exhibits $FD$ as a weak colimit of $FT$ in $\mathcal{C}$.
\end{proof}
\end{proposition}

 Even though ordinary left adjoints preserve (weak) colimits and ordinary right adjoints preserve (weak) limits, the left-weak adjunctions and the right-weak adjunctions do not necessarily verify both these two properties. It is the case that left-weak left adjoints preserve weak colimits and that right-weak right adjoints preserve weak limits (see Proposition \ref{r-w r adjoint preserves w l} and Proposition \ref{l-w l adjoint preserves w c}), but a priori we cannot say anything about left-weak right adjoint and right weak left adjoints.

However, as our main aim is to work with a weakened notion of adjunction, we do not deviate from this principle if we simply require left-weak right adjoints to preserve at least weak finite products and right-weak left adjoints to preserve at least weak finite coproducts: these properties are verified by ordinary left and right adjoints.

\end{document}